\documentclass{amsart}
\usepackage{amsmath,amssymb,amsthm,hyperref,amsxtra}
\usepackage{amsfonts,amsmath,mathrsfs,amssymb,graphicx,multirow,amsthm,latexsym,xcolor}
\usepackage[mathscr]{euscript}
\usepackage[all]{xy}
\usepackage{enumerate}
\usepackage{multicol}
\usepackage{mathtools}
\usepackage{longtable}
\usepackage{lscape}
 	
\usepackage{hyperref}

\usepackage{fullpage}
\allowdisplaybreaks

\newcounter{counterA}
\renewcommand{\thecounterA}{\arabic{counterA}}
\newcommand{\counterA}{\refstepcounter{counterA}%
           \thecounterA}

\newcommand{\ZZ}{\mathbb Z}

\newcommand{\NN}{\mathbb N}

\newcommand{\depth}[1]{\operatorname{depth} {#1}}

\newcommand{\ldb}{\mathopen{[\![}} %left double brackets
\newcommand{\rdb}{\mathclose{]\!]}} %right double brackets
\newcommand{\m}{\mathfrak{m}}

\DeclareMathOperator{\Hom}{Hom}
\DeclareMathOperator{\Ext}{Ext}
\DeclareMathOperator{\pd}{pd}
\DeclareMathOperator{\id}{id}

\theoremstyle{plain}
\newtheorem{theorem}{Theorem}
\newtheorem{prop}[theorem]{Proposition}
\newtheorem{cor}[theorem]{Corollary}
\newtheorem{lemma}[theorem]{Lemma}

\newtheorem{claim}[]{Claim}

\newtheorem{question}[theorem]{Question}

\theoremstyle{definition}
\newtheorem{definition}[theorem]{Definition}
\newtheorem{notation}[theorem]{Notation}

\newtheorem{example}[theorem]{Example}
\newtheorem{remark}[theorem]{Remark}

\newtheorem{obs}[theorem]{Observation}

\numberwithin{theorem}{section}
\numberwithin{equation}{section}
\numberwithin{equation}{theorem}

\begin{document}

\title{Semidualizing Modules over Numerical Semigroup Rings}
\author[Ela Celikbas]{Ela Celikbas}
\address{Ela Celikbas \\ School of Mathematical and Data Sciences, West Virginia University, Morgantown, WV 26506, USA.}
\email{ela.celikbas@math.wvu.edu}

\author[Hugh Geller]{Hugh Geller}
\address{Hugh Geller \\ School of Mathematical and Data Sciences, West Virginia University, Morgantown, WV 26506, USA.}
\email{hugh.geller@mail.wvu.edu}

\author[Toshinori Kobayashi]{Toshinori Kobayashi}
\address{Toshinori Kobayashi \\ School of Science and Technology, Meiji University, 1-1-1 Higashi-Mita, Tama-ku, Kawasaki-shi, Kanagawa 214-8571, Japan.}
\email{tkobayashi@meiji.ac.jp}


\thanks{Toshinori Kobayashi was partly supported by JSPS Grant-in-Aid for JSPS Fellows 21J00567.} 

\begin{abstract}
A semidualizing module is a generalization of Grothendieck's dualizing module. For a local Cohen-Macaulay ring $R$, the ring itself and its canonical module are always realized as (trivial) semidualizing modules. Reasonably, one might ponder the question; when do nontrivial examples exist? In this paper, we study this question in the realm of numerical semigroup rings and completely classify which of these rings with multiplicity at most 9 possess a nontrivial semidualizing module. Using this classification, we construct numerical semigroup rings in any multiplicity at least 9 possesses a nontrivial semidualizing module.
\end{abstract}

\maketitle

\section{Introduction}

The theory of dualizing modules, established by Grothendieck in 1961 (see \cite{Ha}), has been substantially developed and applied in commutative algebra and algebraic geometry. 
Almost ten years later, the concept of semidualizing modules which is a generalization of dualizing modules was first initiated by Foxby as a special case of PG-modules given in \cite{F72}. Semidualizing modules were also independently discovered by Vasconcelos, Golod, and Wakamatsu (see \cite{G84, V74, W88}). Vasconcelos referred them as spherical modules and Golod called them suitable modules. Wakamutsu was interested in the bimodules ${}_{B} T_{A}$ where $A$ and $B$ are commutative local and Artin rings satisfying two properties. In the case of $A=B$, these properties would match the notion of semidualizing modules. Christensen generalized the techniques of homological algebra on complexes studied by Avramov, Foxby, and Golod and introduced the concept of semidualizing complexes, which led to the foundation for the use of the terminology ``semidualizing modules", see \cite{C}. Since then, significant research has been conducted on the homological properties, homological dimensions, characterizations of semidualizing modules in the literature, see \cite{A, CSW, FSW1, FSW2, Ger1, Ger2, HJ, NSW, SW07, TW}, and others. 

For a commutative ring $R$, we say that a finitely generated $R$-module $C$ is \textit{semidualizing} if the natural homothety map $R \rightarrow \Hom_R(C, C)$ is an isomorphism and $\Ext_R^i(C,C)=0$ for all $i>0$.
It is evident from this definition that $R$ itself is a semidualizing $R$-module.
Furthermore, if $R$ is local Cohen--Macaulay with a canonical module $K$, then $K$ is also a semidualizing $R$-module. In 1985, Golod \cite{G85} posed a question asking about the existence of a local ring $R$ that has a nontrivial semidualizing module, i.e., a semidualizing module different from $R$ and $K$. 
In 1987, Foxby provided the first example of rings with nontrivial semidualizing modules in his unpublished notes.

Over a Gorenstein ring R, semidualizing modules are trivial because every semidualizing $R$-module is isomorphic to $R$, see \cite[(8.6)]{C}. Several known results demonstrate that some other local rings also have only trivial semidualizing modules. For example, Sather-Wagstaff , and O. Celikbas and Dao showed that only trivial semidualizing modules exist over determinantal rings (\cite{SW07}) and Veronese subrings (\cite{CD}), respectively. Sanders also showed in \cite{S} that, under mild assumptions, the ring of invariants has only trivial semidualizing modules.
Moreover it is known that all Cohen--Macaulay rings with minimal multiplicity have no nontrivial semidualizing modules, see \cite[Example 4.2.14]{SW10}.
This suggests that the existence of nontrivial semidualizing modules can serve as a valuable criterion to estimate how good the rings are.

A remarkable structure theorem on local rings that have nontrivial semidualizing modules was given by Jorgensen, Leuschke and Sather-Wagstaff \cite{JLSW}.
However, it is not easy to determine from their criteria whether the conditions hold for a particular ring to produce examples with nontrivial semidualizing modules. Notably, even in the case of a ring with Krull dimension one, determining the existence of semidualizing modules becomes a challenging task.

In this paper, we investigate the existence of nontrivial semidualizing modules by restricting our attention to one-dimensional Cohen--Macaulay rings with small multiplicities.
We first observe that if the ring has multiplicity at most $8$, then it has only trivial semidualizing modules (see Proposition \ref{mult8}).
Therefore, it is appropriate to focus on rings of multiplicity $9$.

As our next step, we narrow our focus to numerical semigroup rings, which represent a prototypical class of one-dimensional rings.
Each such ring is associated with an additive submonoid of $\mathbb{N}$.
Numerical semigroup rings share good properties such as being integral domains, Cohen--Macaulayness, and the existence of a canonical module.
Over these rings, many numerical invariants including multiplicities and Cohen--Macaulay types are easy to compute.
To continue our investigation, we examine $R=k[\![t^9,t^{10},t^{11},t^{12},t^{15}]\!]$, which is a numerical semigroup ring of multiplicity 9 given in \cite{GTTT}.
By combining the results of \cite{GTTT} with \cite[Lemma 4.6]{HW}, it is easy to verify the vanishing of $\Ext^1_R(I,I)$ for the fractional ideal $I = (1,t)$.
We then present an argument in Example~\ref{ex7.3}, demonstrating that $I$ is a nontrivial semidualizing module over $R$.
This particular example is noteworthy as we are currently unaware of any prior efforts in the literature regarding the study of semidualizing modules over numerical semigroup rings. This leads us to pose the following question.

\begin{question}\label{Question2} Can we give a characterization for numerical semigroup rings $k \ldb H \rdb$ with $e(k \ldb H \rdb)=9$ that have a nontrivial semidualizing module? 
\end{question}

To address this question, we employ Kunz's concept of classifying numerical semigroups with a fixed multiplicity. Kunz introduces the polyhedron $P_m$ and
gives a bijection $\mu$ between the set of numerical semigroups containing an integer $m \geq 3$ and the set of lattice points of the polyhedron $P_m$. 
In our case, as we are working with multiplicity 9, we are interested in utilizing the polyhedron $P_9$ for an explicit definition. For further information and more details, see Section \ref{section3} or \cite{K}. 

Our first main theorem, which is stated next, gives a complete classification of which numerical semigroup rings of multiplicity 9  have nontrivial semidualizing modules.

\begin{theorem}[{Theorem \ref{mainthm}}]
A numerical semigroup ring $k \ldb H \rdb$ of multiplicity $9$ has a nontrivial semidualizing module if and only if $\mu(H)$ belongs to one of the interiors of faces $F_1,\dots,F_{24}$ listed in Table~\ref{table4}.
\end{theorem} 

In our proof of Theorem~\ref{mainthm}, we use Kunz’s bijection $\mu$ to construct a bijection between families of numerical semigroups and the faces on Kunz’s polyhedron. We show in Proposition~\ref{kunz2.3} that if one member of the family has a non-trivial semidualizing module, then all members of the family have one as well. 
In Table~\ref{table2}, we identify all relevant faces of $P_9$ and a sample member of the corresponding family. 

We then identify which samples of Table~\ref{table2} are Burch (see Definition~\ref{defburch}). We do this since Corollary~\ref{burchsemi} shows that Burch rings only possess the trivial semidualizing modules. This information is assembled in Table~\ref{table3}, which shows that, up to automorphisms of $\mathbb Z/9\mathbb Z$, there are only 7 cases that we need to investigate. Example~\ref{ex7.3}, Proposition~\ref{sample36}, and Proposition~\ref{p45} indicate that 6 of these cases have nontrivial semidualizing modules.

We note that not only does this theorem completely classify numerical semigroups of multiplicity 9 exhibiting nontrivial semidualizing modules, but it also provides the foundation for our other two main theorems.

\begin{theorem}[{Theorem \ref{cor:allmult}}]
For $9 \leq a \in \mathbb{N}$, there exists a numerical semigroup ring $k \ldb H \rdb$ with $e(k \ldb H \rdb) = a$ such that $k \ldb H \rdb$ has a nontrivial semidualizing module.
\end{theorem}

We extend the above theorem one step further with our final main theorem. The comprehensive statement of this theorem (see Theorem~\ref{thm:semidualglue}) constructs a nontrivial semidualizing module for a numerical semigroup ring  resulting from a gluing. However, for conciseness, we present an abridged version below.

\begin{theorem}[{Theorem \ref{thm:semidualglue}}]
If the numerical semigroup ring $k \ldb H \rdb$ is obtained from a gluing of the numerical semigroup rings $k \ldb H_1 \rdb$ and $k \ldb H_2 \rdb$ (see Definition \ref{def:gluing}), then one can construct a semidualizing module for $k \ldb H \rdb$ from the semidualizing modules of $k \ldb H_1 \rdb$ and $k \ldb H_2 \rdb$. Moreover, if either $k \ldb H_1 \rdb$ or $k \ldb H_2 \rdb$ has a nontrivial semidualizing module, then this construction is also nontrivial.
\end{theorem}

The organization of the paper is as follows. In Section~\ref{prelim}, we give a brief exposition of semidualizing modules, Burch ideals, fractional ideals, and numerical semigroup rings.
We also list several facts used throughout the paper and observe that the numerical semigroup ring $R=k[\![ t^9, t^{10}, t^{11}, t^{12}, t^{15} ]\!]$ given in \cite[Example 7.3]{GTTT} possesses a nontrivial semidualizing module; see Example~\ref{ex7.3}. In Section~\ref{section3}, we recall some definitions and the fundamental concept of Kunz's polyhedron introduced in \cite{K}. 

Section~\ref{section4} systemically reduces the study of infinitely many numerical semigroup rings down to the 7 cases we previously mentioned. This is done by identifying the relevant faces of Kunz’s polyhedron and then using results from Section~\ref{prelim} to further eliminate most of those faces. We analyze each of these cases individually to obtain our main result; see Theorem~\ref{mainthm}.

Finally, in Section~\ref{section5}, we introduce the notion of gluings of numerical semigroup rings. We use these gluings to give a constructive proof of Theorem~\ref{cor:allmult}. This is followed by Theorem~\ref{thm:semidualglue} which can be used to identify a nontrivial semidualizing module for any numerical semigroup ring obtained from Theorem~\ref{cor:allmult}. To demonstrate both these theorems, we provide Example \ref{ex:highmult} to exhibit nontrivial semidualizing modules for numerical semigroup rings with multiplicity 10 up to multiplicity 18. We conclude by asking Question \ref{question:monomial}, originally posed to us by Herzog, regarding the generators of semidualizing modules over numerical semigroup rings.

\section{Preliminaries} \label{prelim}

Throughout this paper, we adopt the following notation unless otherwise specified. 

\begin{notation}
Let $(R,\m)$ be a local ring. We denote the multiplicity of $R$ by $e(R)$ and the embedding dimension of $R$ by $\mathrm{edim}(R)$. If $R$ is Cohen--Macaulay, we use $r(R)$ to denote the (Cohen--Macaulay) type of $R$.  For an $R$-module $M$, $\mu_R(M)$ denotes the cardinality of a system of minimum generators of $M$, $\ell_R(M)$ denotes the length of $M$, and $M^{\oplus n}$ is the direct sum of $n$-copies of $M$ for an integer $n\ge 1$. 
\end{notation}

In this section, we give basic definitions and known results utilized in this paper concerning semidualizing modules and numerical semigroup rings. Furthermore, we provide background information on Burch ideals and fractional ideals, which will be relevant for later sections. 

\subsection{Semidualizing Modules} 

Let $(R, \m)$ be a Cohen-Macaulay local ring with a canonical module $K$. 

\begin{definition}
A finitely generated $R$-module $C$ is \emph{semidualizing} if it satisfies the following: 

\begin{enumerate}[ \rm (a)]
\item The natural homothety map $\chi_C^R:R \rightarrow \Hom_R(C, C)$ which is given by $\chi_C^R(r)(c)=rc$ is an $R$-module isomorphism; and 
\item $\Ext_R^i(C,C)=0$ for all $i>0$. 
\end{enumerate}

\end{definition} 

It follows from this definition that the $R$-modules $R$ and $K$ are semidualizing. \\

We say that $R$ has a {\it nontrivial semidualizing module} $C$ if $C\not \cong R$ and $C\not \cong K$. We refer to modules $C\cong R$ or $C\cong K$ as {\it trivial}. Note that $C\cong K$ if and only if $C$ is semidualizing and has finite injective dimension. Also $C\cong R$ if and only if $C$ is semidualizing and has finite projective dimension.

\begin{remark}

Over a Gorenstein ring R, semidualizing modules are trivial because every semidualizing $R$-module is isomorphic to $R$, see \cite[(8.6)]{C}. \end{remark}

We state here some basic properties of semidualizing modules and their canonical duals.

\begin{lemma} \label{dual}
Let $R$ be a Cohen--Macaulay local ring with a canonical module $K$.
Write the functor $\Hom_R(-,K)$ as $(-)^\vee$.
Let $C$ be a semidualizing module over $R$.
Then 
\begin{enumerate}[ \rm(1)]
\item $C^\vee$ is semidualizing.
\item $C \otimes_R C^\vee \cong K$.
\item $\mu_R(C)\mu_R(C^\vee)=r(R)$.
\item $C\cong R$ if and only if $\mu_R(C)=1$.
\item $C\cong K$ if and only if $\mu_R(C)=r(R)$.
\item $C$ is trivial if and only if $C^\vee$ is trivial.
\item If $r(R)$ is a prime number (e.g. $r(R) \le 3$), then $C$ must be trivial.
\end{enumerate}
\end{lemma}

\begin{proof}
(1): It follows from \cite[Corollary (2.12)]{C}.

(2): It follows from \cite[Theorem 3.1]{Ger2}; see also \cite[Fact 2.4]{JLSW}.

(3): This can be established by computing $\mu_R$ of both sides of (2).

(4): Since $C$ is faithful (\cite[Proposition 1.4.4]{SW10}), the assertion follows immediately.

(5),(6),(7): These assertions follow by parts (3) and (4).
\end{proof}

In the rest of this subsection, we give a couple of necessary conditions on local rings to have nontrivial semidualizing modules.
We begin with an observation about Artinian rings.

\begin{remark} \label{socle}
If $(R,\mathfrak{m})$ is an Artinian local ring such that its socle $\mathrm{soc}(R)$ is not contained in $\mathfrak{m}^2$, then it is easy to see that $R$ is isomorphic to a ring of the form $S \ltimes k$ for some local ring $S$. 
It follows from \cite[Corollary 3.5]{NY} that $R$ has only trivial semidualizing modules.
\end{remark}

The following result shows that local rings of small embedding dimension have only trivial semidualizing modules.

\begin{prop}[{\cite{NSW}, c.f. \cite[Proposition 4.10]{LM}}] \label{ecodepth}
Let $(R,\mathfrak{m})$ be a local ring with $\mathrm{edim}(R)-\depth(R)\le 3$.
Then $R$ has only trivial semidualizing modules.
\end{prop}

The next theorem is a result of \cite[Proposition 4.2]{FSW2} and \cite[Fact]{CSW} which uses \cite[Proposition 5.7]{C} and \cite[Theorem 4.5 and Theorem 4.9]{FSW1}.

\begin{theorem}\label{regular}
Let $(R, \m)$ be a Cohen-Macaulay complete local ring and let $\mathbf{x}=x_1, \ldots, x_n \in \m$. If $\mathbf{x}$ is $R$-regular on $R$, then $R/(\mathbf{x})$ has only trivial semidualizing module if and only if $R$ has only trivial semidualizing module. 
\end{theorem}

We also record a result about faithfully flat extensions.

\begin{theorem}[{\cite[Theorem 4.5 and Theorem 4.9]{FSW1}}] \label{fflat}
Let $R \to S$ be a faithfully flat local homomorphism between Cohen--Macaulay local rings.
If $S$ has only trivial semidualizing module, then so does $R$.
\end{theorem}

We then give a necessary condition for Cohen--Macaulay local rings of multiplicity at most $9$ to have nontrivial semidualizing modules.

\begin{prop} \label{mult8}
Let $(R,\mathfrak{m})$ be a Cohen--Macaulay local ring.
\begin{enumerate}[ \rm(1)]
\item If $e(R) \le 8$, then $R$ has only trivial semidualizing modules.
\item If $e(R)=9$ and $R$ has a nontrivial semidualizing module, then $R$ has type $r(R)=4$ and embedding dimension $\mathrm{edim}(R)=4+\dim R$.
\end{enumerate}
\end{prop}

\begin{proof}
Assume $R$ is a Cohen--Macaulay local ring having a nontrivial semidualizing module.
Thanks to Theorem \ref{fflat}, we may pass to a suitable faithfully flat extension so that $R$ is assumed to be complete with an infinite residue field.
Take a minimal reduction $(\mathbf{x})$ of $\mathfrak{m}$.
Then $A:=R/(\mathbf{x})$ is an Artinian ring with  $\ell_A(A)=e(R)$.
By Theorem \ref{regular}, $A$ has a nontrivial semidualizing module.
Note that $r(A)=\ell_A(\mathrm{soc}(A))$ and $\mathrm{edim}(A)=\ell_A(\mathfrak{m}_A/\mathfrak{m}_A^2)$, where $\mathrm{soc}(A)$ is the socle of $A$ and $\mathfrak{m}_A$ is the maximal ideal of $A$.
By Lemma \ref{dual} and Proposition \ref{ecodepth}, we have $\ell_A(\mathrm{soc}(A)) \ge 4$ and $\ell_A(\mathfrak{m}/\mathfrak{m}^2)\ge 4$.
Also, $\mathrm{soc}(A)$ is contained in $\mathfrak{m}_A^2$ by Remark \ref{socle}.
This yields inequalities
\[
8 \le \ell_A(\mathrm{soc}(A))+\ell_A(\mathfrak{m}/\mathfrak{m}^2) \le \ell_A(\mathfrak{m}^2)+\ell_A(\mathfrak{m}/\mathfrak{m}^2) = \ell_A(\mathfrak{m})=\ell_A(A)-1=e(R)-1.
\]
In case (1), i.e. $e(R)\le 8$, a contradiction is derived from the above inequalities.

In case (2), i.e. $e(R)=9$, the above inequalities become equalities, which means that $4=\ell_A(\mathrm{soc}(A))=r(R)$ and $4=\mathrm{edim}(A)$.
The latter one induces $\mathrm{edim}(R)=4+\dim R$.
\end{proof}

\subsection{Burch Ideals} 

The notion of Burch rings were introduced in \cite{DKT} and it was inspired by the results of Burch in \cite{B}.  

\begin{definition} \label{defburch}
Let $(R, \m , k)$ be a local ring.
\begin{enumerate}[ \rm(1)]
\item A \emph{Burch ideal} $I$ is an ideal with $\m I \neq \m(I:_R \m)$.
\item $R$ is called a \emph{Burch ring} if there exist a regular local ring $S$, a Burch ideal $I$ of $S$, and a regular sequence $\mathbf{x}$ of $\widehat{R}$, where $\widehat{R}$ denotes the $\m$-adic completion of $R$, such that $\widehat{R}/\mathbf{x}\widehat{R}$ is isomorphic to $S/I$.
\end{enumerate}
\end{definition}

\begin{remark} \label{burchideal}
If $S$ is a local ring and $I$ is a Burch ideal of $S$, then $S/I$ is a Burch ring; see \cite[Remark~2.9]{DKT}. 
\end{remark}

In order to show that Burch rings only possess trivial semidualizing modules, we adopt the following results on Ext over $R$ by modifying their Tor-counterparts in \cite{DKT}. Specifically, the following result is the Ext-version of \cite[Proposition 7.12]{DKT}. The proof is almost exactly the same but is provided for context.

\begin{prop}\label{prop:BurchExt}
Let $(R, \m, k)$ be a Burch ring of depth zero, and let $M$, $N$ be finitely generated $R$-modules. If $\Ext_R^{\ell}(M,N) = 0 = \Ext_R^{\ell + 1}(M,N)$ for some $\ell \geq 3$, then either $M$ is $R$-free or $N$ has finite injective dimension.
\end{prop}

\begin{proof}
For a finitely generated $R$-module and integer $n\ge 0$, we denote by $\Omega^n M$ the $n$-th syzygy module in a minimal free resolution of $M$. Since all modules are assumed finitely generated, we may assume that $R$ is complete. If $M$ is $R$-free, then we have our result, so we assume $M$ is nonfree. Since $\depth{R} = 0$, $M$ then has infinite projective dimension. By \cite[Lemma 7.4]{DKT}, we get a short exact sequence
\[\xymatrix{
0 \ar[r] & \left(\Omega M\right)^{\oplus n} \ar[r] & X \ar[r] & M^{\oplus n} \ar[r] & 0
}\]
where $n$ is the minimal number of generators of $\m$ and the entries of the minimal presentation matrix of $X$ generates $ \m$. Since $R$ is Burch, \cite[Proposition 4.2]{DKT} tells us that $k$ is a direct summand of $\Omega^2 X$. One can then construct a (not necessarily minimal) free resolution of $X$ over $R$ from the minimal resolutions of $\Omega M$ and $M$ over $R$ to then induce the exact sequence
\[\xymatrix{
0 \ar[r] & \left(\Omega^3 M \right)^{\oplus n} \ar[r] & \Omega^2 X \oplus F \ar[r] & \left(\Omega^2 M\right)^{\oplus n} \ar[r] & 0
}\]
where $F$ is a free $R$-module. Using the long exact sequences in Ext, we have \[\Ext_R^{\ell - 2}\left(\left(\Omega^2 M\right)^{\oplus n},N\right) \cong \left( \Ext_R^{\ell - 2} \left( \Omega^2 M, N\right) \right)^{\oplus n} \cong \left( \Ext_R^{\ell} \left(M, N \right) \right)^{\oplus n} = 0. \] Similarly, we have \[\Ext_R^{\ell - 2}\left(\left(\Omega^3 M\right)^{\oplus n},N\right) \cong \left( \Ext_R^{\ell - 2} \left( \Omega^3 M, N\right) \right)^{\oplus n} \cong \left( \Ext_R^{\ell + 1} \left(M, N \right) \right)^{\oplus n} = 0. \] This then implies that
$0 = \Ext_R^{\ell - 2}\left( \Omega^2 X \oplus F, N \right)\cong \Ext_R^{\ell - 2}\left( \Omega^2 X, N \right).$

Since $k$ is a direct summand of $\Omega^2 X$, we deduce that $\Ext_R^{\ell - 2} (k, N)$ vanishes and thus see that $N$ must have finite injective dimension.
\end{proof}

The following corollary is the Ext version of \cite[Corollary 7.13]{DKT} with the proof based on that of \cite[Corollary 6.6]{NT}.

\begin{cor}\label{cor:BurchExt}
Let $(R,\m,k)$ be a Cohen-Macaulay Burch ring of depth $t$. Let $M$ and $N$ be finitely generated $R$-modules and set $s := t - \depth M$. Assume that there exists an integer $\ell \geq \max\{3,t + 1\}$ such that $\Ext_R^i(M,N) = 0$ for all $\ell + s \leq i \leq \ell + s + t + 1$. Then either $M$ has finite projective dimension or $N$ has finite injective dimension.
\end{cor}

\begin{proof}
As all modules are finitely generated, we may assume $R$ is complete and thus has a canonical module. From this, we apply \cite[Theorem 1.1]{AB} to get the exact sequence $0 \to Y \to X \to N \to 0$ where $X$ is maximal Cohen-Macaulay and $Y$ has finite injective dimension.

We set $L := \Omega_R^s M$ and note that $L$ is maximal Cohen-Macaulay. Moreover, we have $\Ext_R^i(L,N) \cong \Ext_R^{i + s}(M,N) = 0$ for all $\ell \leq i \leq \ell + t + 1$. We also note that the finite injective dimension of $Y$ with the maximal Cohen-Macaulay property of $L$ yields $\Ext_R^i(L,Y)=0$ for all $i > 0$. Combining the Ext-vanishing with the short exact sequence above, we obtain $\Ext_R^i(L,X) = 0$ for $\ell \leq i \leq \ell + t + 1$.

Now, since $R$ is Burch of depth $t$, there exists an $R$-regular sequence $\mathbf{x} = x_1, \ldots, x_t$ such that $R / ( \mathbf{x} ) \cong S/I$ where $S$ is a regular local ring and $I$ is Burch. Since $L$ is maximal Cohen-Macaulay, we also have $\mathbf{x}$ is a regular sequence on $L$ and thus the long exact sequences in Ext derived from the exact sequences
\[\xymatrix{
0 \ar[r] & \frac{L}{(x_1,\ldots,x_{j-1})L} \ar[r]^-{\cdot x_j} & \frac{L}{(x_1,\ldots,x_{j-1})L} \ar[r] & \frac{L}{(x_1,\ldots,x_{j-1},x_j)L} \ar[r] & 0
}\]
for $1 \leq j \leq t$ yields $\Ext_R^i(L/\mathbf{x}L,X)=0$ for $\ell + t \leq i \leq \ell + t + 1$. We next use \cite[section 18, lemma 2(i)]{M} to obtain $\Ext_{R/(\mathbf{x})}^i(L/\mathbf{x}L,X/\mathbf{x}X) = 0$ for all $\ell \leq i \leq \ell + 1$. Applying Proposition \ref{prop:BurchExt}, we find that either $L/\mathbf{x}L$ is $R/(\mathbf{x})$-free or $X/\mathbf{x}X$ has finite injective dimension. In the latter case, since $R/(\mathbf{x})$ is a depth 0, Cohen-Macaulay ring, this means $X/\mathbf{x}X$ is $R/(\mathbf{x})$-injective. As such, \cite[lemma 1.3.5 and corollary 3.1.15]{BH} gives $\pd_R L < \infty$ or $\id_R X < \infty$.

In the case that $\pd_R L < \infty$, we immediately have $\pd_R M < \infty$. 
In the case that $\id_R X < \infty$, in view of the exact sequence $0 \to Y \to X \to N \to 0$, we obtain $\id_R N < \infty$.

\end{proof}

The following is a consequence of Corollary \ref{cor:BurchExt}. 

\begin{cor} \label{burchsemi}
If $I$ is an $\m$-primary Burch ideal of a local ring $R$, then $R/I$ has only trivial semidualizing modules.
\end{cor}

\subsection{Fractional ideals}

Let $A$ be a noetherian ring and let $Q(A)$ be its total ring of fractions.
An $A$-submodule $I$ of $Q(A)$ is said to be a \emph{fractional ideal} if there exists a regular element $a\in R$ such that $aI \subseteq R$.
If $A$ is an integral domain, then every nonzero ideal of $A$ is a fractional ideal.

\begin{remark} \label{fraciso}
For fractional ideals $I$ and $J$, we denote by $I:J$ the colon fractional ideal $I:_{Q(A)}J=\{a\in Q(A)\mid aJ\subseteq I\}$.
Then there is a canonical isomorphism $\Hom_A(J,I) \cong I \colon J$ of $A$-modules; see \cite[Lemma 2.4.2]{HS}.
Using this isomorphism, one can see that the image of the canonical map
\[
J \otimes_A \Hom_A(J,I) \to I
\]
which sends $x\otimes f$ to $f(x)$ for $x\in J$ and $f\in \Hom_A(J,I)$ coincides to $J(I:J)$.
\end{remark}

The next proposition plays a key role in our computation (Example \ref{ex7.3} and Proposition \ref{p45}).

\begin{prop}[{\cite[Proposition 6.3]{GTTT}}] \label{prop6.3} Let $A$ be a Cohen-Macaulay local ring with $\dim A=1$. Assume $I$ and $J$ are both fractional ideals of $A$ where $\mu_A(I)=2$. Then $(J:I)/(A:I)J \cong T(I\otimes_A J)$ as $A$-modules where $T(I\otimes_A J)$ denotes the torsion part of the tensor product $I\otimes_A J$.
\end{prop}

We give one more remark about behaviors of semidualizing modules over local integral domains.

\begin{remark} \label{r214}
Assume $A$ is a local integral domain.
Let $C$ be a semidualizing $A$-module.
Then $C$ is isomorphic to an ideal $I$ of $A$ (\cite[Proposition 3.1]{SW07}).
Assume further that $A$ is Cohen--Macaulay with a canonical module $K$.
Then $K$ can be chosen to be a fractional ideal of $A$.
By Lemma \ref{dual} and Remark \ref{fraciso}, it follows that $I(K:I)=K$.
\end{remark}

\begin{prop}[{\cite[Proposition 3.1]{BV}}] \label{p219}
Let $A$ be a Cohen--Macaulay local ring with $\dim A=1$.
Assume $I$ is a fractional ideal and $K$ is a canonical ideal.
Then the canonical map $A \to \Hom_A(I,I)$ is an isomorphism if and only if $I(K:I)=K$.
\end{prop}

\subsection{Numerical Semigroup Rings}

Let $\NN$ denote the set of non-negative integers. A \emph{numerical semigroup} $H$ is a subset of $\NN$ containing 0, closed under addition and satisfying $\gcd(H)=1$, where $\gcd$ denotes the greatest common divisor. Note that $\gcd(H)=1$ is equivalent to saying that $\NN \setminus H$ is finite.

The \emph{Ap\'ery set} of $H$ with respect to $c\in H$, denoted $\mathrm{Ap}_{c}(H)$, is the set of smallest elements in $H$ in each congruence class modulo $c$.
That is, $\mathrm{Ap}_{c}(H)=\{h_0, h_1, \ldots, h_{c-1}\}$ where $h_0=0$ and $h_i=\min \{ h\in H \mid h \equiv i \mod c \}$.

Let $0<a_1<a_2<\cdots< a_{\ell}$ be integers such that $\gcd(a_1, a_2, \ldots, a_{\ell})=1$ and consider the numerical semigroup $H=\langle a_1,a_2,\dots,a_{\ell}\rangle$ generated by $a_1,a_2,\dots,a_{\ell}$.
Let $k[\![t]\!]$ be a formal power series ring over a field $k$, and let $R_H=k[\![H]\!]=k[\![t^{a_1}, t^{a_2}, \ldots, t^{a_{\ell}}]\!]$ be the subring of $k[\![t]\!]$. Then $R_H$ is called a \emph{numerical semigroup ring} over $k$.

Note that $R_H$ is a one-dimensional Cohen-Macaulay complete local integral domain with maximal ideal $\m=(t^{a_1}, t^{a_2}, \ldots, t^{a_{\ell}})$, a canonical module $K$, and multiplicity $e(R_H)=a_1$.
We also denote by $A_H$ the factor ring $R_H/t^{a_1}R_H$.

Our goal is to answer Question~\ref{Question2} for the rest of the paper. First we observe that the numerical semigroup ring $R=k[\![ t^9, t^{10}, t^{11}, t^{12}, t^{15} ]\!]$ given in \cite[Example 7.3]{GTTT} has a nontrivial semidualizing module.

\begin{example} \label{ex7.3} Let $H=\langle 9, 10, 11, 12, 15\rangle $. Then $R:=R_H=k[\![ t^9, t^{10}, t^{11}, t^{12}, t^{15} ]\!]$. Here we have $K=(1, t, t^3, t^4)$.  Let $I=(1,t)$. Then $I^\vee=K:I=(1, t^3)$. Note that $\mu_R(I)=\mu_R(I^\vee)=2$, and $\mu_R(K)=4$. One can check that $R:I=(t^9, t^{10}, t^{11})$ and $I^\vee:I=(t^9, t^{10}, t^{11}, t^{12}, t^{13}, t^{14})$. Hence we have $(R:I)I^\vee=I^\vee:I$ which implies $I\otimes_R I^{\vee}$ is torsion-free by Proposition~\ref{prop6.3}, and we get $\Ext_R^1(I, I)=0$ by \cite[Lemma 4.6]{HW}. Now we show $\Ext^2_R(I, I)=0$. Since $I^{\vee}=(1, t^3)$,  one can apply Proposition~\ref{prop6.3} one more time to get $$T((R:I)\otimes_R I^{\vee})=((R:I):I^{\vee})/(R:I^{\vee})(R:I).$$
Furthermore we have  $(R:I):I^{\vee}=(t^{18}, t^{19}, t^{20}, t^{21}, t^{22}, t^{23}, t^{24}, t^{25}, t^{26})$ and $R:I^{\vee}=(t^9, t^{12}, t^{15} )$. Therefore $(R:I):I^{\vee}=(R:I^{\vee})(R:I)$ which implies that $(R:I)\otimes_R I^{\vee}$ is torsion-free.
As $\mu_R(I)=2$, applying \cite[Lemma 3.3]{HH}, we get a short exact sequence
\[
0 \to R:I \to R^{\oplus 2} \to I \to 0.
\]
This yields that $\Ext^2_R(I, I) \cong \Ext^1_R(R:I, I)$.
Now we get $\Ext^2_R(I, I)=0$ by \cite[Lemma 4.6]{HW}.
Next note that we have the short exact sequence 
\begin{equation} 0\xrightarrow{} I \xrightarrow{t^9} R:I \xrightarrow{} R/ ((t^9, t^{10}):_R t^{11}) \xrightarrow{} 0 
\end{equation}
where $(t^9, t^{10}):_R t^{11}=R:I$. Thus we get the long exact sequence 
\begin{equation}\cdots \xleftarrow{} \Ext_R^i(I, I)\xleftarrow{} \Ext_R^i(R:I, I) \xleftarrow{} \Ext^i_R(R/(R:I), I) \xleftarrow{} \cdots 
\end{equation}

Since $\Ext_R^i(R:I, I) \cong \Ext_R^{i+1} (I,I)$ and $\Ext^i_R(R/(R:I), I) \cong \Ext^{i-1}_R(R:I, I) \cong \Ext^i_R(I, I)$ for all $i\geq 2$, we have $\Ext^i (I, I) = 0$ for all $i\ge 1$.

Finally we have $\Hom_R(I,I)=R$ by Proposition \ref{p219}.
Therefore the numerical semigroup ring $R=k[\![ t^9, t^{10}, t^{11}, t^{12}, t^{15} ]\!]$ has a nontrivial semidualizing module and $e(R)=9$. 
\end{example} 

\section{Kunz's polyhedron} \label{section3}

We use basic terminology and facts from convex geometry, which we summarize here.
For more detailed treatments, we refer the reader to \cite{BG}. 

An \emph{affine half-space} of $\mathbb{R}^n$ is a subset $\{x\in \mathbb{R}^n \mid \lambda(x)\ge \alpha\}$ for some $\lambda \in \Hom_{\mathbb{R}}(\mathbb{R}^n,\mathbb{R})$ and $\alpha\in \mathbb{R}^n$.
The half-space is \emph{linear} if $\alpha=0$.
A \emph{polyhedron} is the intersection of finitely many affine half-spaces, whereas a \emph{cone} is the intersection of finitely many linear half-spaces.
A \emph{support hyperplane} of a polyhedron $P$ is a hyperplane $H$ such that $P$ is contained in one of two closed half-spaces into which $\mathbb{R}^n$ is decomposed by $H$.
A \emph{face} of $P$ is a subset $F=P\cap H$ where $H$ is a support hyperplane of $P$.
The \emph{dimension} of $F$ is the dimension of its affine hull.
A \emph{facet} is a face $F$ with $\dim F=\dim P-1$.
A \emph{vertex} is a face $F$ with $\dim F=0$.
We denote by $F^{\circ}$ the interior of $F$.

Next we recall the fundamental concept of the polyhedron which was introduced by Kunz; see \cite{K} for more details.

Fix an integer $m \ge 3$.
Let $\mathfrak{H}_m$ be the set of all numerical semigroups containing $m$.
For fixed $H\in \mathfrak{H}_m$ and each $i=1,\dots,m-1$, let $h_i$ be the smallest element of $H$ such that $h_i \equiv i \mod m$, so that $\{0,h_1,\dots,h_{m-1}\}$ forms the Ap\'ery set $\mathrm{Ap}_m(H)$ of $H$ with respect to $m$.
We also write $h_i=i+\mu_im$, where $\mu_i$ is a non-negative integer.
Since $H$ is a disjoint union of $m\mathbb{N},h_1+m\mathbb{N},\dots, h_{m-1}+m\mathbb{N}$, $H$ is uniquely determined by
$(\mu_1,\dots,\mu_{m-1}) \in \mathbb{N}^{m-1}$.
The point $(\mu_1,\dots,\mu_{m-1})$ is a solution of the system of linear inequalities
\begin{equation} \label{ieq}
\left\{\begin{array}{ll}
X_i+X_j \ge X_{i+j} & (1 \le i < j \le m-1, i+j<m),\\
X_i+X_j \ge X_{i+j-m}-1 & (1 \le i < j \le m-1, i+j>m).
\end{array}\right.
\end{equation}

Similarly, $(h_1,\dots,h_{m-1})$ is a solution of the system of linear inequalities
\begin{equation} \label{ieqq}
\left\{\begin{array}{ll}
X_i+X_j \ge X_{i+j} & (1 \le i < j \le m-1, i+j<m),\\
X_i+X_j \ge X_{i+j-m} & (1 \le i < j \le m-1, i+j>m).
\end{array}\right.
\end{equation}

Let $P_m$ (resp. $C_m$) be the solution set of \eqref{ieq} (resp. \eqref{ieqq}) in $\mathbb{R}^{m-1}$.
Then $P_m$ is a polyhedron with unique vertex $v:=(-\frac{1}{m},-\frac{2}{m},\dots,-\frac{m-1}{m})$, and the translation $x \mapsto x-v$ maps $P_m$ bijectively onto the polyhedral cone $C_m$.
The correspondence $H \mapsto (\mu_1,\dots,\mu_{m-1})$ gives a bijection
\[
\mu \colon \mathfrak{H}_m \to P_m \cap \mathbb{N}^{m-1}.
\]

For a pair $(i,j)$ of integers such that $1 \le i \le j \le m-1$ with $i+j\not=m$, let $E_{ij}$ be the hyperplane

\begin{equation} \label{eq}
E_{ij}:=\left\{\begin{array}{ll}
X_i+X_j = X_{i+j} & (\text{if }i+j<m),\\
X_i+X_j = X_{i+j-m}-1 & (\text{if }i+j>m).
\end{array}\right.
\end{equation}

These hyperplanes correspond bijectively to facets of $P_m$.
Consequently, a face $F$ of $P_m$ is uniquely determined by the collection $\Delta_F :=\{(i,j) \mid F \subseteq E_{ij}\}$.

The $k$-vector space $R_H/t^mR_H$ is spanned by images $\bar{t}^{h_i}$ of $t^{h_i}$ and $1$.
The multiplication on the $k$-algebra $R_H/t^mR_H$ are determined by the formulas: 
\begin{equation} \label{mult}
\bar{t}^{h_i}\bar{t}^{h_j}=\bar{t}^{h_{i+j}} \text{ if and only if }\mu(H) \in E_{ij} \text{, and } \text{if }\bar{t}^{h_i}\bar{t}^{h_j}\not=\bar{t}^{h_{i+j}} \text{, then }\bar{t}^{h_i}\bar{t}^{h_j}=0.
\end{equation}
This leads the following observation:
\begin{prop}[{\cite[Beispiele 2.4a]{K}}] %Kunz
Let $F$ be a face of $P_m$, and let $H, H' \in \mathfrak{H}_m$.
If $\mu(H)$ and $\mu(H')$ belong to the interior $F^{\circ}$ of $F$, then $R_H/t^mR_H$ and $R_{H'}/t^mR_{H'}$ are isomorphic as graded $k$-algebras.
\end{prop}

Elements of $\mathrm{Aut}(\mathbb{Z}/m\mathbb{Z})$, as permutations of variables $\{X_1,X_2,\dots,X_{m-1}\}$, give automorphisms of the set of all faces of $C_m$.
Thus, via the translation from $C_m$ to $P_m$, the group $\mathrm{Aut}(\mathbb{Z}/m\mathbb{Z})$ acts on the set of all faces of $P_m$.
Moreover, faces linked by this action have some similarities:

\begin{prop}
\label{kunz2.3}
Let $\sigma$ be an element of $\mathrm{Aut}(\mathbb{Z}/m\mathbb{Z})$ and let $F$ be a face of $P_m$.
Let $H$ and $H'$ be two semigroups such that $\mu(H)\in F^{\circ}$ and $\mu(H')\in \sigma(F)^{\circ}$.
\begin{enumerate}[ \rm (1)]
\item $R_H/t^mR_H \cong R_{H'}/t^mR_{H'}$ as $k$-algebras.
\item $R_H$ possesses a nontrivial semidualizing module if and only if so does $R_{H'}$.
\end{enumerate}
\end{prop}

\begin{proof}
(1): It follows from \cite[Proposition 2.3]{K}. %Kunz
For the reader's convenience, we give a sketch of the proof.
Let $\{0,h_1,\dots,h_{m-1}\}$ and $\{0,h'_1,\dots,h'_{m-1}\}$ be the Ap\'ery sets of $H$ and $H'$ respectively.
Define a $k$-linear map $\varphi\colon R_H/t^mR_H \to R_{H'}/t^mR_{H'}$ which sending $1$ to $1$ and $\bar{t}^{h_i}$ to $\bar{t}^{h'_{\sigma(i)}}$ for each $i$.
Then we see that $\varphi$ is a $k$-algebra homomorphism by \eqref{mult} and the following equivalences:
\begin{align*}
\bar{t}^{h_i}\bar{t}^{h_j}=\bar{t}^{h_{i+j}} & \iff \mu(H)\in E_{ij} & (\text{by \eqref{mult}})\\
& \iff F \subseteq E_{ij} & (\text{by the assumption } \mu(H)\in F^{\circ})\\
& \iff \sigma(F)\subseteq \sigma(E_{ij}) & (\text{since $\sigma$ preserves inclusions of faces})\\
& \iff \sigma(F)\subseteq E_{\sigma(i)\sigma(j)} & (\text{since $\sigma(E_{ij})=E_{\sigma(i)\sigma(j)}$})\\
& \iff \mu(H') \in E_{\sigma(i)\sigma(j)}
& (\text{by the assumption } \mu(H')\in \sigma (F)^{\circ})\\
& \iff \bar{t}^{h'_{\sigma(i)}}\bar{t}^{h'_{\sigma(j)}}=\bar{t}^{h'_{\sigma(i)+\sigma(j)}} & (\text{by \eqref{mult}})\\
& \iff\bar{t}^{h'_{\sigma(i)}}\bar{t}^{h'_{\sigma(j)}}=\bar{t}^{h'_{\sigma(i+j)}} & (\text{since }\sigma(i)+\sigma(j)=\sigma(i+j))
\end{align*}
Obviously, $\varphi$ is a bijection.
Thus $R_H/t^mR_H$ and $R_{H'}/t^mR_{H'}$ are isomorphic as $k$-algebras.

(2): This is a consequence of (1) by using Theorem \ref{regular}.
\end{proof}

\section{Classification of faces in multiplicity $9$}\label{section4}

In this section, our objective is to determine all faces $F$ of $P_9$ for which every semigroup rings $R_H$ with $\mu(H)\in F^{\circ}$ possess a nontrivial semidualizing module.
To see $F$ has such a property, it is enough to check that there exists $H\in\mathfrak{H}_9$ such that $\mu(H)\in F^{\circ}$ and $A_H$ has a nontrivial semidualizing module.

Let $F$ be one such face, and $H$ be a semigroup such that $\mu(H)\in F^{\circ}$.
According to Proposition \ref{mult8}, $R_H$ needs to have multiplicity $9$, embedding dimension $5$, and type $4$.
Therefore, $A_H$ is a $9$-dimensional local $k$-algebra having embedding dimension $4$ and type $4$.
Take a quadruple $(a,b,c,d)$ of integers such that $1 \le a <b <c <d \le 8$ and $\bar{t}^{h_a},\bar{t}^{h_b},\bar{t}^{h_c},\bar{t}^{h_d}$ form a system minimal generators of the maximal ideal of $A_H$.
Note that $(a,b,c,d)$ is uniquely determined by $F$.
Let $\sigma\in\mathrm{Aut}(\mathbb{Z}/9\mathbb{Z})$.
In view of Proposition \ref{kunz2.3}(1), if $(a,b,c,d)$ is the quadruple of $F$, then up to order, $(\sigma(a),\sigma(b),\sigma(c),\sigma(d))$ is the quadruple of $\sigma(F)$.

As we observe in Proposition \ref{kunz2.3}(2), it suffices to examine a single face from each orbit under the action of $\mathrm{Aut}(\mathbb{Z}/9\mathbb{Z})$.
Therefore, we can limit our focus to one representative $(a,b,c,d)$ per each orbits under this group action.
The following table exhibits how $\mathrm{Aut}(\mathbb{Z}/9\mathbb{Z})=\{1,2,4,5,7,8\}$ acts on $(a,b,c,d)$:

\begin{longtable}[h]{|c|c|c|c|c|c|}
\caption{}
\label{table}
\\
\hline
1 & 2 & 4 & 5 & 7 & 8
\endfirsthead
\hline
$(1,2,3,4)$ & $(2,4,6,8)$ & $(3,4,7,8)$ & $(1,2,5,6)$ & $(1,3,5,7)$ & $(5,6,7,8)$\\
\hline
$(1,2,3,5)$ & $(1,2,4,6)$ & $(2,3,4,8)$ & $(1,5,6,7)$ & $(3,5,7,8)$ & $(4,6,7,8)$ \\
\hline
$(1,2,3,6)$ & $(2,3,4,6)$ & $(3,4,6,8)$ & $(1,3,5,6)$ & $(3,5,6,7)$ & $(3,6,7,8)$ \\
\hline
$(1,2,3,7)$ & $(2,4,5,6)$ & $(1,3,4,8)$ & $(1,5,6,8)$ & $(3,4,5,7)$ & $(2,6,7,8)$ \\
\hline
$(1,2,3,8)$ & $(2,4,6,7)$ & $(3,4,5,8)$ & $(1,4,5,6)$ & $(2,3,5,7)$ & $(1,6,7,8)$ \\
\hline
$(1,2,4,5)$ & $(1,2,4,8)$ & $(2,4,7,8)$ & $(1,2,5,7)$ & $(1,5,7,8)$ & $(4,5,7,8)$ \\
\hline
$(1,2,4,7)$ & $(2,4,5,8)$ & $(1,4,7,8)$ & $(1,2,5,8)$ & $(1,4,5,7)$ & $(2,5,7,8)$ \\
\hline
$(1,2,6,7)$ & $(2,3,4,5)$ & $(1,4,6,8)$ & $(1,3,5,8)$ & $(2,5,6,7)$ & $(2,3,7,8)$ \\
\hline
$(1,2,6,8)$ & $(2,3,4,7)$ & $(4,5,6,8)$ & $(1,3,4,5)$ & $(2,5,6,7)$ & $(1,3,7,8)$ \\
\hline
$(1,2,7,8)$ & $(2,4,5,7)$ & $(1,4,5,8)$ & $(1,4,5,8)$ & $(2,4,5,7)$ & $(1,2,7,8)$ \\
\hline
$(1,3,4,6)$ & $(2,3,6,8)$ & $(3,4,6,7)$ & $(2,3,5,6)$ & $(1,3,6,7)$ & $(3,5,6,8$ \\
\hline
$(1,3,4,7)$ & $(2,5,6,8)$ & $(1,3,4,7)$ & $(2,5,6,8)$ & $(1,3,4,7)$ & $(2,5,6,8)$ \\
\hline
$(1,3,6,8)$ & $(2,3,6,7)$ & $(3,4,5,6)$ & $(3,4,5,6)$ & $(2,3,6,7)$ & $(1,3,6,8)$ \\
\hline
$(1,4,6,7)$ & $(2,3,5,8)$ & $(1,4,6,7)$ & $(2,3,5,8)$ & $(1,4,6,7)$ & $(2,3,5,8)$ \\
\hline
\end{longtable}

\begin{obs}
Denote by $\mathfrak{m}_{A_H}$ the maximal ideal of $A_H$.
Assume $A_H$ has a nontrivial semidualizing module.
Then,
\begin{enumerate}[ \rm (i)]
\item Since $\bar{t}^{h_a},\bar{t}^{h_b},\bar{t}^{h_c},\bar{t}^{h_d}$ are minimal generators of $\mathfrak{m}_{A_H}$,
a pair $(i,j)$ with $i+j\in\{a,b,c,d\}$ does not satisfy $\bar{t}^{h_i}\bar{t}^{h_j}=\bar{t}^{h_{i+j}}$. 
So if $(i,j) \in \Delta_F$, then $i+j \not \in \{a,b,c,d\}$.
\item For each $r\in \{1,\dots, 8\}\setminus\{a,b,c,d\}$, $\bar{t}^{h_r}$ is not a part of minimal generators of $\mathfrak{m}_{A_H}$.
In other words, $\bar{t}^{h_r}$ belongs to $\mathfrak{m}_{A_H}^2$.
Thus $\Delta_F$ must contain a pair $(i,j)$ such that $i+j \equiv r \mod 9$.
\item Observe that $\ell_{A_H}(\mathrm{soc}(A_H))=r(R)=4$ and 
\[
\ell_{A_H}(\mathfrak{m}_{A_H}^2)=\ell_{A_H}(A_H) -\ell_{A_H}( \mathfrak{m}_{A_H}/\mathfrak{m}_{A_H}^2)-\ell_{A_H} (A_H/\mathfrak{m}_{A_H})=9-4-1=4.
\]
Since $\mathrm{soc}(A_H) \subseteq \mathfrak{m}_{A_H}^2$ (Remark \ref{socle}), $\mathrm{soc}(A_H)$ must be equal to $\mathfrak{m}_{A_H}^2$, and so $\bar{t}^{h_r}$ belongs to $\mathrm{soc}(A_H)$ for each $r\in\{1,\dots,8\}\setminus\{a,b,c,d\}$.
In particular, for any $s \in \{1,\dots,8\}$, $\bar{t}^{h_s}\bar{t}^{h_r}=0$ in $A_H$, i.e. $(r,s),(s,r)\not\in \Delta_F$.
It means that for each $(i,j) \in \Delta_F$, both $i$ and $j$ are in $\{a,b,c,d\}$.

\item Additionaly, as $\mathrm{soc}(A_H) \subseteq \mathfrak{m}_{A_H}^2$, $\bar{t}^{h_a},\bar{t}^{h_b},\bar{t}^{h_c},\bar{t}^{h_d}$ are not in $\mathrm{soc}(A_H)$.
This means that for each $s\in \{a,b,v,d\}$, there exists $(i,j)$ such that either $s=i$ or $s=j$.
\end{enumerate}
\end{obs}

In summary, $\Delta_F$ must satisfy the following rules:
\begin{enumerate}[ \rm(R1)]
\item for each $(i,j) \in \Delta_F$, $i+j$ is not equivalent to any of $a,b,c,d$ modulo $9$.
\item for each $(i,j) \in \Delta_F$, both $i$ and $j$ are in $\{a,b,c,d\}$.
\item for each $r\in \{1,\dots, 8\}\setminus\{a,b,c,d\}$, there exists $(i,j)\in \Delta_F$ such that $i+j \equiv r \mod 9$.
\item for each $s\in \{a,b,c,d\}$, there exists $(i,j)\in \Delta_F$ such that either $s=i$ or $s=j$.
\end{enumerate}

We also give some necessary conditions on $\Delta$ to coincide $\Delta_F$ for some face $F$.

\begin{lemma}\label{lemmafaces}
For a face $F$, $\Delta_F$ must satisfy the following rules:
\begin{enumerate}[ \rm(R1)]
\setcounter{enumi}{4}

\item If $(1,2), (2,6), (6,7) \in \Delta_F$, then $(1,7), (2,2), (6,6) \in \Delta_F$.

\item If $(1,7), (2,2), (6,6) \in \Delta_F$, then $(1,2), (2,6), (6,7) \in \Delta_F$.

\item $(1,4),(1,7),(4,4),(7,7)$ cannot simultaneously belong to $\Delta_F$.

\item $(1,4), (1,7), (4,7)$ cannot simultaneously belong to $\Delta_F$.

\item $(1,1), (4,4), (7,7)$ cannot simultaneously belong to $\Delta_F$.

\end{enumerate}
\end{lemma}

\begin{proof}
Let $H$ be a semigroup such that $\mu(H)\in F^{\circ}$.
Let $\{0,h_1,\dots,h_8\}$ be the Ap\'ery set of $H$.

(R5): Assume $(1,2),(2,6),(6,7)\in \Delta_F$.
It follows that $h_1+h_2=h_3 \le 2h_6$, $h_2+h_6=h_8 \le h_1+h_7$, and $h_6+h_7=h_4 \le 2h_2$.
Taking a sum of all these inequalities, we have a same value $h_2+2h_2+2h_6+h_7$ in both left and right hand sides.
Thus, all of the inequalities must be equalities.
This means that $(1,7),(2,2),(6,6)\in \Delta_F$.

(R6): This can be seen by similar argument to (R5).

(R7): Assume $(1,4),(1,7),(4,4),(7,7) \in \Delta_F$.
It follows that $2h_4=h_8=h_1+h_7$ and $2h_7=h_5=h_1+h_4$.
Therefore we have $3(h_7-h_4)=0$.
As $h_4\not=h_7$, this leads a contradiction.

(R8): Assume $(1,4),(1,7),(4,7)\in \Delta_F$.
It follows that $h_1+h_4=h_5 \le 2h_7$, $h_1+h_7=h_8 \le 2h_4$, and $h_4+h_7=h_2 \le 2h_1$.
Taking a sum of all these inequalities, we have a same value $2h_1+2h_4+2h_7$ in both left and right hand sides.
Thus, all of the inequalities must be equalities.
This means that $(1,1),(4,4),(7,7)\in \Delta_F$.
Due to (R7), this cannot occur.

(R9): This can be seen by similar argument to (R5).
\end{proof}

The following table gives a list all the subsets $\Delta \subseteq \{(i,j)\mid 1 \le i \le j \le 8, i+j\not=9\}$ satisfying the rules (R1)-(R4) above, where each $(a,b,c,d)$ is an entry of the first column of Table \ref{table}:

{
\begin{longtable}[h]{|c|c|c|c|}
\caption{}
\label{table2}
\\
\hline
No. & $(a,b,c,d)$ & $\Delta$ & Sample $H$
\endfirsthead
\hline

$\counterA{}$ & $(1,2,3,4)$ & $\{(1,4),(2,4),(3,4),(4,4)\}$ & $\langle 9, 13, 19 ,20, 21\rangle$\\

\hline

$\counterA{}$ & $(1,2,3,4)$ & $\{(1,4),(3,3),(3,4),(4,4)\}$ & $\langle 9, 12 ,13, 19, 29\rangle$\\

\hline

$\counterA{}$ & $(1,2,3,4)$ & $\{(1,4),(2,3),(2,4), (3,4), (4,4)\}$ & $\langle 9, 11 ,13, 19, 21\rangle$ \\

\hline

$\counterA{}$ & $(1,2,3,4)$ & $\{(1,4),(2,3),(3,3), (3,4), (4,4)\}$ & $\langle 9, 12 ,13, 19, 20\rangle$ \\

\hline

$\counterA{}$ & $(1,2,3,4)$ & $\{(1,4),(2,4),(3,3), (3,4), (4,4)\}$ & $\langle 9, 13 ,19, 21, 29\rangle$ \\

\hline

$\counterA{}$ & $(1,2,3,4)$ & $\{(1,4),(2,3),(2,4),(3,3), (3,4), (4,4)\}$ & $\langle 9, 10 ,11, 12, 13\rangle$ \\

\hline

$\counterA{}$ & $(1,2,3,5)$ & $\{(1,3),(1,5),(2,5),(3,5)\}$ & $\langle 9, 14 ,19, 21, 29\rangle$ \\

\hline

$\counterA{}$ & $(1,2,3,5)$ & $\{(1,3),(3,3),(2,5),(3,5)\}$ & $\langle 9, 12 ,14, 19, 20\rangle$ \\

\hline

$\counterA{}$ & $(1,2,3,5)$ & $\{(2,2),(1,5),(2,5),(3,5)\}$ & $\langle 9, 14 ,19, 20, 30\rangle$ \\

\hline

$\counterA{}$ & $(1,2,3,5)$ & $\{(1,3),(1,5),(3,3),(2,5),(3,5)\}$ & $\langle 9, 19 ,21, 23, 29\rangle$ \\

\hline

$\counterA{}$ & $(1,2,3,5)$ & $\{(1,3),(2,2),(1,5),(2,5),(3,5)\}$ & $\langle 9, 14 ,19, 20, 21\rangle$ \\

\hline

$\counterA{}$ & $(1,2,3,5)$ & $\{(1,3),(2,2),(3,3),(2,5),(3,5)\}$ & $\langle 9, 12 ,20, 23, 28\rangle$ \\

\hline

$\counterA{}$ & $(1,2,3,5)$ & $\{(2,2),(1,5),(3,3),(2,5),(3,5)\}$ & $\langle 9, 23 ,29, 30, 37\rangle$ \\

\hline

$\counterA{}$ & $(1,2,3,5)$ & $\{(1,3),(2,2),(1,5),(3,3),(2,5),(3,5)\}$ & $\langle 9, 10, 11, 12, 14\rangle$ \\

\hline

$\counterA{}$ & $(1,2,3,6)$ & $\{(1,3),(2,3),(1,6),(2,6)\}$ & $\langle 9, 12, 15, 19, 20\rangle$ \\

\hline

$\counterA{}$ & $(1,2,3,6)$ & $\{(2,2),(2,3),(1,6),(2,6)\}$ & $\langle 9, 11, 15, 19, 21\rangle$ \\

\hline

$\counterA{}$ & $(1,2,3,6)$ & $\{(1,3),(2,2),(2,3),(1,6),(2,6)\}$ & $\langle 9, 10, 11, 12, 15\rangle$ \\

\hline

$\counterA{}$ & $(1,2,3,7)$ & $\{(1,3),(2,3), (3,3),(1,7)\}$ & $\langle 9, 12, 19, 25, 29\rangle$ \\

\hline

$\counterA{}$ & $(1,2,3,7)$ & $\{(2,2),(2,3),(3,3),(1,7)\}$ & $\langle 9, 20, 21, 25, 28\rangle$ \\

\hline

$\counterA{}$ & $(1,2,3,7)$ & $\{(2,2),(7,7),(3,3),(1,7)\}$ & $\langle 9, 16, 20, 21, 28\rangle$ \\

\hline

$\counterA{}$ & $(1,2,3,7)$ & $\{(1,3),(2,3),(7,7),(3,3),(1,7)\}$ & $\langle 9,12, 16, 19, 20\rangle$ \\

\hline

$\counterA{}$ & $(1,2,3,7)$ & $\{(1,3),(2,2),(2,3),(3,3),(1,7)\}$ & $\langle 9,10, 11, 12, 16\rangle$ \\

\hline

$\counterA{}$ & $(1,2,3,7)$ & $\{(1,3),(2,2),(7,7),(3,3),(1,7)\}$ & $\langle 9,16, 19, 20, 21\rangle$ \\

\hline

$\counterA{}$ & $(1,2,3,7)$ & $\{(2,2),(2,3),(7,7),(3,3),(1,7)\}$ & $\langle 9,28, 29, 34, 39\rangle$ \\

\hline

$\counterA{}$ & $(1,2,3,7)$ & $\{(1,3), (2,2),(2,3),(7,7),(3,3),(1,7)\}$ & $\langle 9,21, 25, 29, 37\rangle$ \\

\hline

$\counterA{}$ & $(1,2,3,8)$ & $\{(1,3),(2,3),(3,3),(8,8)\}$ & $\langle 9,12, 17, 19, 20\rangle$ \\

\hline

$\counterA{}$ & $(1,2,3,8)$ & $\{(1,3),(2,2),(2,3),(3,3),(8,8)\}$ & $\langle 9,12, 17, 20, 28\rangle$ \\

\hline

$\counterA{}$ & $(1,2,4,5)$ & $\{(1,2),(1,5),(2,5),(4,4)\}$ & $\langle 9,14, 19, 20, 22\rangle$ \\

\hline

$\counterA{}$ & $(1,2,4,5)$ & $\{(1,2),(2,4),(2,5),(4,4)\}$ & $\langle 9,11, 13, 14, 19\rangle$ \\

\hline

$\counterA{}$ & $(1,2,4,5)$ & $\{(1,2),(1,5),(2,4),(2,5),(4,4)\}$ & $\langle 9,10, 11, 13, 14\rangle$ \\

\hline

$\counterA{}$ & $(1,2,4,7)$ & $\{(1,2),(1,4),(2,4),(1,7)\}$ & $\langle 9,19,20, 31, 34\rangle$ \\

\hline

$\counterA{}$ & $(1,2,4,7)$ & $\{(1,2),(7,7),(2,4),(4,4)\}$ & $\langle 9,20,22,25, 37\rangle$ \\

\hline

$\counterA{}$ & $(1,2,4,7)$ & $\{(1,2),(7,7),(2,4),(1,7)\}$ & $\langle 9,16,19,20, 22\rangle$ \\

\hline

$\counterA{}$ & $(1,2,4,7)$ & $\{(1,2),(1,4),(2,4),(1,7),(4,4)\}$ & $\langle 9,10,11,13,16\rangle$ \\

\hline

$\counterA{}$ & $(1,2,4,7)$ & $\{(1,2),(7,7),(2,4),(1,7),(4,4)\}$ & $\langle 9,20,25,31,37\rangle$ \\

\hline

$\counterA{}$ & $(1,2,4,7)$ & $\{(1,2),(1,4),(7,7),(2,4),(1,7)\}$ & $\langle 9,19,20,25,31\rangle$ \\

\hline

$\counterA{}$ & $(1,2,4,7)$ & $\{(1,2),(1,4),(7,7),(2,4),(4,4)\}$ & $\langle 9,11,13,16,19\rangle$ \\

\hline

$\counterA{}$ & $(1,2,4,7)$ & $\{(1,2),(1,4),(7,7),(2,4),(1,7),(4,4)\}$ & $\emptyset$ \\

\hline

$\counterA{}$ & $(1,2,6,7)$ & $\{(1,2),(2,2),(7,7),(2,6)\}$ & $\langle 9,11,16,24,28\rangle$ \\

\hline

$\counterA{}$ & $(1,2,6,7)$ & $\{(1,2),(6,7),(7,7),(1,7)\}$ & $\langle 9,25,28,38,42\rangle$ \\

\hline

$\counterA{}$ & $(1,2,6,7)$ & $\{(1,2),(6,7),(7,7),(2,6)\}$ & $\emptyset$ \\

\hline

$\counterA{}$ & $(1,2,6,7)$ & $\{(6,6),(2,2),(7,7),(1,7)\}$ & $\emptyset$ \\

\hline

$\counterA{}$ & $(1,2,6,7)$ & $\{(1,2),(2,2),(7,7),(1,7),(2,6)\}$ & $\langle 9,11,16,19,24\rangle$  \\

\hline

$\counterA{}$ & $(1,2,6,7)$ & $\{(1,2),(6,7),(7,7),(1,7),(2,6)\}$ & $\emptyset$ \\

\hline

$\counterA{}$ & $(1,2,6,7)$ & $\{(1,2),(2,2),(6,7),(7,7),(1,7)\}$ & $\langle 9,16,19,20,24\rangle$  \\

\hline

$\counterA{}$ & $(1,2,6,7)$ & $\{(1,2),(2,2),(6,7),(7,7),(2,6)\}$ & $\emptyset$ \\

\hline

$\counterA{}$ & $(1,2,6,7)$ & $\{(6,6),(2,2),(7,7),(1,7),(2,6)\}$ & $\emptyset$ \\

\hline

$\counterA{}$ & $(1,2,6,7)$ & $\{(6,6),(6,7),(7,7),(1,7),(2,6)\}$ & $\langle 9,15,16,19,20\rangle$  \\

\hline

$\counterA{}$ & $(1,2,6,7)$ & $\{(6,6),(2,2),(6,7),(7,7),(1,7)\}$ & $\emptyset$ \\

\hline

$\counterA{}$ & $(1,2,6,7)$ & $\{(1,2),(6,6),(2,2),(7,7),(1,7)\}$ & $\emptyset$ \\

\hline

$\counterA{}$ & $(1,2,6,7)$ & $\{(1,2),(6,6),(2,2),(7,7),(2,6)\}$ & $\langle 9,11,15,16,19\rangle$  \\

\hline

$\counterA{}$ & $(1,2,6,7)$ & $\{(1,2),(6,6),(6,7),(7,7),(1,7)\}$ & $\langle 9,19,24,25,29\rangle$  \\

\hline

$\counterA{}$ & $(1,2,6,7)$ & $\{(1,2),(6,6),(6,7),(7,7),(2,6)\}$ & $\emptyset$  \\

\hline

$\counterA{}$ & $(1,2,6,7)$ & $\{(1,2),(2,2),(6,7),(7,7),(1,7),(2,6)\}$ & $\emptyset$  \\

\hline

$\counterA{}$ & $(1,2,6,7)$ & $\{(1,2),(6,6),(2,2),(7,7),(1,7),(2,6)\}$ & $\emptyset$  \\

\hline

$\counterA{}$ & $(1,2,6,7)$ & $\{(1,2),(6,6),(6,7),(7,7),(1,7),(2,6)\}$ & $\emptyset$  \\

\hline

$\counterA{}$ & $(1,2,6,7)$ & $\{(1,2),(6,6),(2,2),(6,7),(7,7),(1,7)\}$ & $\emptyset$  \\

\hline

$\counterA{}$ & $(1,2,6,7)$ & $\{(1,2),(6,6),(2,2),(6,7),(7,7),(2,6)\}$ & $\emptyset$  \\

\hline

$\counterA{}$ & $(1,2,6,7)$ & $\{(6,6),(2,2),(6,7),(7,7),(1,7),(2,6)\}$ & $\emptyset$  \\

\hline

$\counterA{}$ & $(1,2,6,7)$ & $\{(1,2),(6,6),(2,2),(6,7),(7,7),(1,7),(2,6)\}$ & $\langle 9,19,24,29,34\rangle$  \\

\hline

$\counterA{}$ & $(1,2,6,8)$ & $\{(1,2),(2,2),(6,8),(1,6)\}$ & $\langle 9,19,20,24,26\rangle$  \\

\hline

$\counterA{}$ & $(1,2,6,8)$ & $\{(1,2),(2,2),(6,8),(8,8)\}$ & $\langle 9,17,19,20,24\rangle$  \\

\hline

$\counterA{}$ & $(1,2,6,8)$ & $\{(6,6),(2,2),(6,8),(1,6)\}$ & $\langle 9,15,19,20,26\rangle$  \\

\hline

$\counterA{}$ & $(1,2,6,8)$ & $\{(1,2),(2,2),(6,8),(1,6),(8,8)\}$ & $\langle 9,19,26,29,33\rangle$  \\

\hline

$\counterA{}$ & $(1,2,6,8)$ & $\{(1,2),(6,6),(2,2),(6,8),(1,6)\}$ & $\langle 9,19,24,26,29\rangle$  \\

\hline

$\counterA{}$ & $(1,2,6,8)$ & $\{(1,2),(6,6),(2,2),(6,8),(8,8)\}$ & $\langle 9,17,20,24,28\rangle$  \\

\hline

$\counterA{}$ & $(1,2,6,8)$ & $\{(6,6),(2,2),(6,8),(1,6),(8,8)\}$ & $\langle9,15,17,19,20\rangle$  \\

\hline

$\counterA{}$ & $(1,2,6,8)$ & $\{(1,2),(6,6),(2,2),(6,8),(1,6),(8,8)\}$ & $\langle 9,11,15,17,19\rangle$  \\

\hline

$\counterA{}$ & $(1,2,7,8)$ & $\{(1,2),(2,2),(7,7),(7,8)\}$ & $\langle 9,16,17,19,20\rangle$  \\

\hline

$\counterA{}$ & $(1,3,4,6)$ & $\{(1,1),(1,4),(1,6),(3,4),(4,4)\}$ & $\langle 9,10,12,13,15\rangle$  \\

\hline

$\counterA{}$ & $(1,3,4,7)$ & $\{(1,1),(1,4),(3,3),(1,7)\}$ & $\langle 9,19,21,31,34\rangle$  \\

\hline

$\counterA{}$ & $(1,3,4,7)$ & $\{(1,1),(7,7),(3,3),(4,4)\}$ & $\emptyset$ \\

\hline

$\counterA{}$ & $(1,3,4,7)$ & $\{(4,7),(1,4),(3,3),(1,7)\}$ & $\emptyset$ \\

\hline

$\counterA{}$ & $(1,3,4,7)$ & $\{(4,7),(1,4),(3,3),(4,4)\}$ & $\langle 9,13,21,25,28\rangle$  \\

\hline

$\counterA{}$ & $(1,3,4,7)$ & $\{(4,7),(7,7),(3,3),(1,7)\}$ & $\langle 9,16,21,28,31\rangle$  \\

\hline

$\counterA{}$ & $(1,3,4,7)$ & $\{(1,1),(1,4),(3,3),(1,7),(4,4)\}$ & $\langle 9,10,12,13,16\rangle$  \\

\hline

$\counterA{}$ & $(1,3,4,7)$ & $\{(1,1),(1,4),(7,7),(3,3),(1,7)\}$ & $\langle 9,19,21,25,31\rangle$  \\

\hline

$\counterA{}$ & $(1,3,4,7)$ & $\{(1,1),(1,4),(7,7),(3,3),(4,4)\}$ & $\emptyset$ \\

\hline

$\counterA{}$ & $(1,3,4,7)$ & $\{(1,1),(7,7),(3,3),(1,7),(4,4)\}$ & $\emptyset$ \\

\hline

$\counterA{}$ & $(1,3,4,7)$ & $\{(4,7),(1,4),(3,3),(1,7),(4,4)\}$ & $\emptyset$ \\

\hline

$\counterA{}$ & $(1,3,4,7)$ & $\{(4,7),(1,4),(7,7),(3,3),(1,7)\}$ & $\emptyset$ \\

\hline

$\counterA{}$ & $(1,3,4,7)$ & $\{(4,7),(1,4),(7,7),(3,3),(4,4)\}$ &  $\langle 9,12,13,16,19\rangle$  \\

\hline

$\counterA{}$ & $(1,3,4,7)$ & $\{(4,7),(7,7),(3,3),(1,7),(4,4)\}$ &  $\langle 9,16,21,22,28\rangle$  \\

\hline

$\counterA{}$ & $(1,3,4,7)$ & $\{(1,1),(4,7),(1,4),(3,3),(1,7)\}$ &  $\emptyset$ \\

\hline

$\counterA{}$ & $(1,3,4,7)$ & $\{(1,1),(4,7),(1,4),(3,3),(4,4)\}$ &  $\langle 9,13,19,21,25\rangle$  \\

\hline

$\counterA{}$ & $(1,3,4,7)$ & $\{(1,1),(4,7),(7,7),(3,3),(1,7)\}$ &  $\langle 9,12,16,19,22\rangle$  \\

\hline

$\counterA{}$ & $(1,3,4,7)$ & $\{(1,1),(4,7),(7,7),(3,3),(4,4)\}$ &  $\emptyset$ \\

\hline

$\counterA{}$ & $(1,3,4,7)$ & $\{(1,1),(1,4),(7,7),(3,3),(1,7),(4,4)\}$ &  $\emptyset$ \\

\hline

$\counterA{}$ & $(1,3,4,7)$ & $\{(1,1),(4,7),(1,4),(7,7),(3,3),(1,7)\}$ &  $\emptyset$ \\

\hline

$\counterA{}$ & $(1,3,4,7)$ & $\{(1,1),(4,7),(1,4),(7,7),(3,3),(4,4)\}$ &  $\emptyset$ \\

\hline

$\counterA{}$ & $(1,3,4,7)$ & $\{(1,1),(4,7),(1,4),(3,3),(1,7),(4,4)\}$ &  $\emptyset$ \\

\hline

$\counterA{}$ & $(1,3,4,7)$ & $\{(1,1),(4,7),(7,7),(3,3),(1,7),(4,4)\}$ &  $\emptyset$ \\

\hline

$\counterA{}$ & $(1,3,4,7)$ & $\{(4,7),(1,4),(7,7),(3,3),(1,7),(4,4)\}$ &  $\emptyset$ \\

\hline

$\counterA{}$ & $(1,3,4,7)$ & $\{(1,1),(4,7),(1,4),(7,7),(3,3),(1,7),(4,4)\}$ & $\emptyset$ \\

\hline

$\counterA{}$ & $(1,3,6,8)$ & $\{(1,1),(1,3),(6,8),(1,6)\}$ &  $\langle 9,10,15,17,21\rangle$  \\

\hline

$\counterA{}$ & $(1,3,6,8)$ & $\{(1,1),(1,3),(6,8),(8,8)\}$ &  $\langle 9,17,19,24,30\rangle$  \\

\hline

$\counterA{}$ & $(1,3,6,8)$ & $\{(3,8),(1,3),(6,8),(1,6)\}$ &  $\langle 9,12,15,26,28\rangle$  \\

\hline

$\counterA{}$ & $(1,3,6,8)$ & $\{(3,8),(1,3),(6,8),(8,8)\}$ &  $\langle 9,12,15,17,28\rangle$  \\

\hline

$\counterA{}$ & $(1,3,6,8)$ & $\{(1,1),(1,3),(6,8),(1,6),(8,8)\}$ &  $\langle 9,24,26,28,39\rangle$  \\

\hline

$\counterA{}$ & $(1,3,6,8)$ & $\{(1,1),(3,8),(1,3),(6,8),(1,6)\}$ &  $\langle 9,12,15,19,26\rangle$  \\

\hline

$\counterA{}$ & $(1,3,6,8)$ & $\{(1,1),(3,8),(1,3),(6,8),(8,8)\}$ &  $\langle 9,17,19,21,24\rangle$  \\

\hline

$\counterA{}$ & $(1,3,6,8)$ & $\{(3,8),(1,3),(6,8),(1,6),(8,8)\}$ &  $\langle 9,12,15,17,19\rangle$  \\

\hline

$\counterA{}$ & $(1,3,6,8)$ & $\{(1,1),(3,8),(1,3),(6,8),(1,6),(8,8)\}$ &  $\langle 9,15,17,19,21\rangle$  \\

\hline

$\counterA{}$ & $(1,4,6,7)$ & $\{(1,1),(6,6),(1,4),(1,7)\}$ &  $\langle 9,19,24,31,34\rangle$  \\

\hline

$\counterA{}$ & $(1,4,6,7)$ & $\{(1,1),(6,6),(7,7),(4,4)\}$ &  $\emptyset$ \\

\hline

$\counterA{}$ & $(1,4,6,7)$ & $\{(4,7),(6,6),(1,4),(1,7)\}$ &  $\emptyset$ \\

\hline

$\counterA{}$ & $(1,4,6,7)$ & $\{(4,7),(6,6),(1,4),(4,4)\}$ &  $\langle 9,22,24,34,37\rangle$ \\

\hline

$\counterA{}$ & $(1,4,6,7)$ & $\{(4,7),(6,6),(7,7),(1,7)\}$ &  $\langle 9,24,25,37,40\rangle$ \\

\hline

$\counterA{}$ & $(1,4,6,7)$ & $\{(1,1),(6,6),(1,4),(1,7),(4,4)\}$ &  $\langle 9,15,19,22,25\rangle$ \\

\hline

$\counterA{}$ & $(1,4,6,7)$ & $\{(1,1),(6,6),(1,4),(7,7),(1,7)\}$ &  $\langle 9,15,19,25,31\rangle$ \\

\hline

$\counterA{}$ & $(1,4,6,7)$ & $\{(1,1),(6,6),(1,4),(7,7),(4,4)\}$ &  $\emptyset$ \\

\hline

$\counterA{}$ & $(1,4,6,7)$ & $\{(1,1),(6,6),(7,7),(1,7),(4,4)\}$ & $\emptyset$ \\

\hline

$\counterA{}$ & $(1,4,6,7)$ & $\{(1,1),(4,7),(6,6),(1,4),(1,7)\}$ & $\emptyset$ \\

\hline

$\counterA{}$ & $(1,4,6,7)$ & $\{(1,1),(4,7),(6,6),(1,4),(4,4)\}$ & $\langle 9,13,15,19,25\rangle$  \\

\hline

$\counterA{}$ & $(1,4,6,7)$ & $\{(1,1),(4,7),(6,6),(7,7),(1,7)\}$ & $\langle 9,15,16,19,22\rangle$  \\

\hline

$\counterA{}$ & $(1,4,6,7)$ & $\{(1,1),(4,7),(6,6),(7,7),(4,4)\}$ & $\emptyset$ \\

\hline

$\counterA{}$ & $(1,4,6,7)$ & $\{(4,7),(6,6),(1,4),(1,7),(4,4)\}$ & $\emptyset$ \\

\hline

$\counterA{}$ & $(1,4,6,7)$ & $\{(4,7),(6,6),(1,4),(7,7),(1,7)\}$ & $\emptyset$ \\

\hline

$\counterA{}$ & $(1,4,6,7)$ & $\{(4,7),(6,6),(1,4),(7,7),(4,4)\}$ & $\langle 9,13,15,16,19 \rangle$   \\

\hline

$\counterA{}$ & $(1,4,6,7)$ & $\{(4,7),(6,6),(7,7),(1,7),(4,4)\}$ & $\langle 9,15,16,22,28 \rangle$   \\

\hline

$\counterA{}$ & $(1,4,6,7)$ & $\{(1,1),(6,6),(1,4),(7,7),(1,7),(4,4)\}$ & $\emptyset$  \\

\hline

$\counterA{}$ & $(1,4,6,7)$ & $\{(1,1),(4,7),(6,6),(1,4),(1,7),(4,4)\}$ & $\emptyset$  \\

\hline

$\counterA{}$ & $(1,4,6,7)$ & $\{(1,1),(4,7),(6,6),(1,4),(7,7),(1,7)\}$ & $\emptyset$  \\

\hline

$\counterA{}$ & $(1,4,6,7)$ & $\{(1,1),(4,7),(6,6),(1,4),(7,7),(4,4)\}$ & $\emptyset$  \\

\hline

$\counterA{}$ & $(1,4,6,7)$ & $\{(1,1),(4,7),(6,6),(7,7),(1,7),(4,4)\}$ & $\emptyset$ \\

\hline

$\counterA{}$ & $(1,4,6,7)$ & $\{(4,7),(6,6),(1,4),(7,7),(1,7),(4,4)\}$ & $\emptyset$ \\

\hline

$\counterA{}$ & $(1,4,6,7)$ & $\{(1,1),(4,7),(6,6),(1,4),(7,7),(1,7),(4,4)\}$ & $\emptyset$  \\

\hline
\end{longtable}
}

\begin{remark}[Remark on Table 2]
\begin{enumerate}[ \rm(1)]
\item Due to (R5), $\Delta$ of no.\ 41, 44, 46, 53, 54, 56 and 58 are not derived from some faces $F$.
\item Due to (R6), $\Delta$ of no.\ 42, 47, 49, 50,  55, 57 and 59 are not derived from some faces $F$.
\item Due to (R7), $\Delta$ of no.\ 38, 88, 94, 121,  126 and 127 are not derived from some faces $F$.
\item Due to (R8), $\Delta$ of no.\ 73, 80, 81, 84, 89, 91, 93, 106, 113, 117, 118, 122 and 123 are not derived from some faces $F$.
\item Due to (R9), $\Delta$ of no.\ 72, 78, 79, 87,  90, 92, 105, 111, 112, 116, 124 and 125 are not derived from some faces $F$.
\end{enumerate}
\end{remark}

Next, we investigate the existence of semidualizing modules of $A_H$ via the defining ideal of $A_H$.
Take a system of minimal generators $m_0=9,m_1,m_2,m_3,m_4$ of $H$ such that $m_1<m_2<m_3<m_4$.
Since $A_H$ is a local Artinian ring and generated by four elements $\bar{t}^{m_1},\bar{t}^{m_2},\bar{t}^{m_3},\bar{t}^{m_4}$ as a $k$-algebra, we get a surjection $\phi \colon k[\![x,y,z,w]\!] \to A_H$ of local $k$-algebras.
We set $I_H$ to be the kernel of $\phi$.
Once we know generators of $I_H$, we can verify whether $I_H$ is Burch or not by just examining the inequality of Definition \ref{defburch} for $I_H$.
Remark that if $I_H$ is Burch, then it follows by Corollary \ref{burchsemi} with Remark \ref{burchideal} that $A_H$ has only trivial semidualizing modules.
Our computations are presented in the following table.

\begin{longtable}{|c|c|c|c|}
\caption{}
\label{table3}
\\
\hline
No. & Sample $H$ & $I_H$ & Is $I_H$ Burch?
\endfirsthead

\hline
1 & $\langle 9, 13, 19 ,20, 21\rangle$ & $(x^2,xy,xz,xw^2,y^2,yz,yw^2,z^2,zw^2,w^3)$ & Yes\\

\hline
2 & $\langle 9, 12, 13, 19, 29\rangle$ & $(x^3,x^2y,xy^2,xz,xw,y^3,y^2z,yw,z^2,zw,w^2)$ & Yes\\

\hline
3 & $\langle 9, 11 ,13, 19, 21\rangle$ & $(x^2,xy^2,xyw,xz,xw-yz,y^3,y^2w,z^2,zw,w^2)$ & Yes\\

\hline
4 & $\langle 9, 12 ,13, 19, 20\rangle$ & $(x^3,x^2y,xy^2,xz,xw-yz,y^3,yw,z^2,zw,w^2)$ & Yes\\

\hline
5 & $\langle 9, 13 ,19, 21, 29\rangle$ & $(x^3,x^2y,x^2z,xw-z^2,x^2w,y^2,yz,yw,zw,w^2)$ & Yes\\

\hline
6 & $\langle 9, 10 ,11, 12, 13\rangle$ & $(x^2, xy, xz, yz-xw, xw^2, y^2, z^2-yw, yw^2, zw^2, w^3)$ & Yes\\

\hline
7 & $\langle 9, 14 ,19, 21, 29\rangle$  & $(x^2, xyz, y^2, yw, z^2, zw, w^2)$ & Yes\\

\hline
8 & $\langle 9, 12 ,14, 19, 20\rangle$  & $(x^3, x^2y, x^2z, xw, y^2, yz, z^2, zw, w^2)$ & Yes\\

\hline
9 & $\langle 9, 14 ,19, 20, 30\rangle$  & $(x^2, xz^2, y^2, yz, yw, z^3, zw, w^2)$ & Yes\\

\hline
10 & $\langle 9, 19 ,21, 23, 29\rangle$  & $(x^2, xyz, y^2-xz, xw, yw, z^2, w^2)$ & Yes\\

\hline
11 & $\langle 9, 14 ,19, 20, 21\rangle$  & $(x^2, xyw, y^2, yz, z^2-yw, zw, w^2)$ & Yes\\

\hline
12 & $\langle 9, 12 ,20, 23, 28\rangle$  & $(x^3, x^2z, xy, y^2-xw, yw, z^2, zw, w^2)$ & Yes\\

\hline
13 & $\langle 9, 23 ,29, 30, 37\rangle$  & $(x^2, xy^2, z^2-xw, y^3, yz, yw, zw, w^2)$ & Yes\\

\hline
14 & $\langle 9, 10, 11, 12, 14\rangle$  & $(x^2, xy, y^2-xz, xzw, z^2-xw, yz, w^2)$ & Yes\\

\hline
15 & $\langle 9, 12, 15, 19, 20\rangle$  & $(x^2, xy, y^2, z^2, zw, w^2)$ & No\\

\hline
16 & $\langle 9, 11, 15, 19, 21\rangle$  & $(x^3, x^2y, x^2w, xz, y^2, yw, z^2, zw, w^2)$ & Yes\\

\hline
17 & $\langle 9, 10, 11, 12, 15\rangle$  & $(x^2, xy, y^2-xz, z^2, zw, w^2)$ & No\\

\hline
18 & $\langle 9, 12, 19, 25, 29\rangle$ & $(x^3, x^2y, x^2w, xz, y^2, yw, z^2, zw, w^2)$ & Yes\\

\hline
19 & $\langle 9, 20, 21, 25, 28\rangle$ & $(x^3, x^2y, xy^2, xz, xw, y^3, yz, yw, z^2, w^2)$ & Yes\\

\hline
20 & $\langle 9, 16, 20, 21, 28\rangle$ & $(x^3, x^2w, xy, xz, y^3, yz, yw, z^3, zw, w^2)$ & Yes\\

\hline
21 & $\langle 9,12, 16, 19, 20\rangle$ & $(x^3, x^2z, xy, y^2-xw, yw, z^2, zw, w^2)$ & Yes\\

\hline
22 & $\langle 9,10, 11, 12, 16\rangle$ & $(x^2, xy, y^2-xz, xz^2, yz^2, yw, z^3, zw, w^2)$ & Yes\\

\hline
23 & $\langle 9,16, 19, 20, 21\rangle$ & $(x^3, x^2y, xz, xw, y^2, yz, z^2-yw, zw, w^3)$ & Yes\\

\hline
24 & $\langle 9,28, 29, 34, 39\rangle$ & $(x^2, xy, xw, y^3, yz, z^2-yw, zw, w^3)$ & Yes\\

\hline
25 & $\langle 9,21, 25, 29, 37\rangle$ & $(x^3, xy, y^2-xz, z^2-xw, y^2, yz, zw, w^2)$ & Yes\\

\hline
26 & $\langle 9,12, 17, 19, 20\rangle$ & $(x^3, x^2z, x^2w, xy, y^3, yz, yw, z^2, zw, w^3)$ & Yes\\

\hline
27 & $\langle 9,12, 17, 20, 28\rangle$ & $(x^3, x^2z, x^2w, xy, z^2-xw, y^3, yz, yw, zw, w^2)$ & Yes\\

\hline
28 & $\langle 9,14, 19, 20, 22\rangle$ & $(x^2, xyz, xw, y^2, yw, z^2, zw, w^3)$ & Yes\\

\hline
29 & $\langle 9,11, 13, 14, 19\rangle$ & $(x^2, xy^2, y^3, yz, yw, z^2, zw, w^2)$ & Yes\\

\hline
30 & $\langle 9,10, 11, 13, 14\rangle$ & $(x^2, xz,yz-xw, y^2, z^3, zw, w^2)$ & Yes\\

\hline
31 & $\langle 9,19,20, 31, 34\rangle$ & $(x^2, xyz, y^2, yw, z^2, zw, w^2)$ & Yes\\

\hline
32 & $\langle 9,20,22,25, 37\rangle$ & $(x^2, xy^2, xz, y^3, yz, yw, z^3, zw, w^2)$ & Yes\\

\hline
33 & $\langle 9,16,19,20, 22\rangle$ & $(x^3, x^2y, xz, xw, y^2, yw, z^2, w^2)$ & Yes\\

\hline
34 & $\langle 9,10,11,13,16\rangle$ & $(x^2, xyz, z^2-xw, y^2, yw, zw, w^2)$ & Yes\\

\hline
35 & $\langle 9,20,25,31,37\rangle$ & $(x^2, xy, y^3, yz, z^2-yw, zw, w^2)$ & Yes\\

\hline
36 & $\langle 9,19,20,25,31\rangle$ & $(x^2, xz, z^2-xw, y^2, zw, w^2)$ & No\\

\hline
37 & $\langle 9,11,13,16,19\rangle$ & $(x^2, xy^2, xz, y^3, yz, z^2-yw, zw, w^2)$ & Yes\\

\hline
39 & $\langle 9,11,16,24,28\rangle$ & $(x^3, x^2z, x^2w, xy, y^3, yz, yw, z^2, zw, w^2)$ & Yes \\

\hline
40 & $\langle 9,25,28,38,42\rangle$ & $(x^3, x^2y, x^2w, xz, y^2, yw, z^2, zw, w^2)$ & Yes \\

\hline
43 & $\langle 9,11,16,19,24\rangle$ & $(x^3, x^2z, xy, yz-xw, y^3, yw, z^2, zw, w^2)$ & Yes\\

\hline
45 & $\langle 9,16,19,20,24\rangle$ & $(x^3, x^2y, xz, z^2-xw, y^2, yw, zw, w^2)$ & Yes\\

\hline
48 & $\langle 9,15,16,19,20\rangle$ & $(x^3, x^2y, xy^2, xz, yz-xw, y^3, yw, z^2, zw, w^2)$ & Yes\\

\hline
51 & $\langle 9,11,15,16,19\rangle$ & $(x^3, x^2y, x^2w, xz, y^2-xw, yz, yw, z^3, zw, w^2)$ & Yes\\

\hline
52 & $\langle 9,19,24,25,29\rangle$ & $(x^2, xy, xz^2, y^2-xw, yz^2, yw, z^3, zw, w^2)$ & Yes\\

\hline
60 & $\langle 9,19,24,29,34\rangle$ & $(x^2, xy, y^2-xz, yz-xw, z^2-yw, zw, w^3)$ & Yes\\

\hline
61 & $\langle 9,19,20,24,26\rangle$ & $(x^2, xy^2, xw, y^3, yz, yw, z^2, w^2)$ & Yes\\

\hline
62 & $\langle 9,17,19,20,24\rangle$ & $(x^3, x^2w, xy, xz, y^2, yz^2, yw, z^3, zw, w^2)$ & Yes\\

\hline
63 & $\langle 9,15,19,20,26\rangle$ & $(x^3, x^2y, x^2w, xz, y^2, yz, yw, z^3, zw, w^2)$ & Yes\\

\hline
64 & $\langle 9,19,26,29,33\rangle$ & $(x^2, xy, xz^2, xw-y^2, yz, z^3, zw, w^2)$ & Yes\\

\hline
65 & $\langle 9,19,24,26,29\rangle$ & $(x^2, xz, xw-y^2, yw, z^2, zw, zw, w^3)$ & Yes\\

\hline
66 & $\langle 9,17,20,24,28\rangle$ & $(x^3, xy, x^2z, xw, y^3, yz, yw-z^2, zw, w^2)$ & Yes\\

\hline
67 & $\langle 9,15, 17, 19, 20\rangle$ & $(x^3, x^2y, x^2z, y^2-xz, xw, yz, z^2, zw, w^3)$ & Yes\\
\hline
68 & $\langle 9,11,15,17,19\rangle$ & $(x^3,xy,xz,xw-y^2,yw-z^2,zw,w^2)$ & Yes\\

\hline
69 & $\langle 9,16,17,19,20\rangle$ & $(x^3,x^2y,xz,xw,y^2,yz,yw,z^2,zw^2,w^3)$ & Yes\\

\hline
70 & $\langle 9,10,12,13,15\rangle$ & $(x^3,xy,x^2z,xz^2,xw-yz,y^2,yw,z^3,zw,w^2)$ & Yes\\

\hline
71 & $\langle 9,19,21,31,34\rangle$ & $(x^3,xy,x^2z,x^2w,y^3,yz,yw,z^2,zw,w^2)$ & Yes\\

\hline
74 & $\langle 9,13,21,25,28\rangle$ & $(x^3,xy,x^2z,x^2w,y^3,yz,yw,z^2,zw,w^2)$ & Yes\\

\hline
75 & $\langle 9,16,21,28,31\rangle$ & $(x^3,xy,x^2z,x^2w,y^3,yz,yw,z^2,zw,w^2)$ & Yes\\

\hline
76 & $\langle 9,10,12,13,16\rangle$ & $(x^3,xy,x^2z,xw-z^2,x^2w,y^3,yz,yw,zw,w^2)$ & Yes\\

\hline
77 & $\langle 9,19,21,25,31\rangle$ & $(x^3,xy,x^2z,xw-z^2,x^2w,y^3,yz,yw,zw,w^2)$ & Yes\\

\hline
82 & $\langle 9,12,13,16,19\rangle$ & $(x^3,xy,xz,xw,y^3,y^2z,yw-z^2,y^2w,zw,w^2)$ & Yes\\

\hline
83 & $\langle 9,16,21,22,28\rangle$ & $(x^3,xy,x^2z,xw-z^2,x^2w,y^3,yz,yw,zw,w^2)$ & Yes\\

\hline
85 & $\langle 9,13,19,21,25\rangle$ & $(x^3,x^2y,xz,xw-y^2,x^2w,yz,yw,z^3,zw,w^2)$ & Yes\\

\hline
86 & $\langle 9,12,16,19,22\rangle$ & $(x^3,xy,xz,xw,y^3,y^2z,yw-z^2,y^2w,zw,w^2)$ & Yes\\

\hline
95 &  $\langle 9,10,15,17,21\rangle$ & $(x^3,x^2y,xz,x^2w,y^2,yw,z^2,zw,w^2)$ & Yes\\

\hline
96 &  $\langle 9,17,19,24,30\rangle$ & $(x^3,xy,x^2z,xw,y^3,yz,y^2w,z^2,zw,w^2)$ & Yes\\

\hline
97 &  $\langle 9,12,15,26,28\rangle$ & $(x^2,xy,y^2,z^2,zw,w^2)$ & No\\

\hline
98 &  $\langle 9,12,15,17,28\rangle$ & $(x^2,xy,xz^2,y^2,yz^2,yw,z^3,zw,w^2)$ & Yes\\

\hline
99 &  $\langle 9,24,26,28,39\rangle$ & $(x^2,xz-y^2,xw,yz,yw,z^3,z^2w,w^2)$ & Yes\\

\hline
100 &  $\langle 9,12,15,19,26\rangle$ & $(x^2,xy,xw-z^2,y^2,zw,w^2)$ & No\\

\hline
101 &  $\langle 9,17,19,21,24\rangle$ & $(x^3,xy,xz-y^2,yw,z^2,zw,w^2)$ & Yes\\

\hline
102 &  $\langle 9,12,15,17,19\rangle$ & $(x^2,xy,y^2,yw-z^2,zw,w^2)$ & No\\

\hline
103 &  $\langle 9,15,17,19,21\rangle$ & $(x^2,xz-y^2,xw,yz,yw-z^2,w^2)$ & No\\

\hline
104 &  $\langle 9,19,24,31,34\rangle$ & $(x^3,xy,x^2z,x^2w,y^3,yz,yw,z^2,zw,w^2)$ & Yes\\

\hline
107 &  $\langle 9,22,24,34,37\rangle$ & $(x^3,xy,x^2z,x^2w,y^3,yz,yw,z^2,zw,w^2)$ & Yes\\

\hline
108 &  $\langle 9,24,25,37,40\rangle$ & $(x^3,xy,xz,xw,y^3,y^2z,y^2w,z^2,zw,w^2)$ & Yes\\

\hline
109 &  $\langle 9,15,19,22,25\rangle$ & $(x^3,xy,xz,xw,y^3,y^2z,yw-z^2,y^2w,zw,w^2)$ & Yes\\

\hline
110 &  $\langle 9,15,19,25,31\rangle$ & $(x^3,xy,xz,xw,y^3,y^2z,yw-z^2,y^2w,zw,w^2)$ & Yes\\

\hline
114 & $\langle 9,13,15,19,25\rangle$ & $(x^3,xy,x^2z,xw-z^2,x^2w,y^3,yz,yw,zw,w^2)$ & Yes\\

\hline
115 & $\langle 9,15,16,19,22\rangle$ & $(x^3,xy,xz,xw,y^3,y^2z,yw-z^2,y^2w,zw,w^2)$ & Yes\\

\hline
119 & $\langle 9,13,15,16,19 \rangle$ & $(x^3,xy,x^2z,xw-z^2,x^2w,y^3,yz,yw,zw,w^2)$ & Yes\\

\hline
120 & $\langle 9,15,16,22,28 \rangle$ & $(x^3,xy,xz,xw,y^3,y^2z,yw-z^2,y^2w,zw,w^2)$ & Yes\\

\hline
\end{longtable}

What remains to be established is whether the semigroup rings $R_H$ have a nontrivial semidualizing module or not, where $A_H$ is not Burch.
By Table \ref{table3}, such an $H$ comes from one of no.\ 15, 17, 36, 97, 100, 102, or 103.
We already saw in Example \ref{ex7.3} that $R_H$ for the Sample $H$ of no.\ 17 has a nontrivial semidualizing module.
The others are investigated bellow.

\begin{prop}\label{sample36}
Let $H=\langle9,19,20,25,31\rangle$ (the Sample of no.\ 36).
Then $R_H=k[\![t^9,t^{19},t^{20},t^{25},t^{31}]\!]$ has only trivial semidualizing modules.
\end{prop}

\begin{proof}
Let $v \colon Q(k[\![t]\!]) \to \mathbb{Z} \cup\{\infty\}$ be the valuation associated to the discrete valuation ring $k[\![t]\!]$ with $v(t)=1$.

We recall the basic properties of $v$.
For $x,y\in Q(k[\![t]\!])$, we have the following:
\begin{enumerate}
\item $v(xy)=v(x)+v(y)$.
\item $v(x+y) \ge \min\{v(x),v(y)\}$.
\item If $v(x)<v(y)$, then $v(x+y)=v(x)$.
\item If $x\in k\setminus\{0\}$, then $v(x)=0$.
\end{enumerate}
It is also easy to verify that $v(R_H)=H$.

Note that $K=(1,t,t^7,t^{12})$ is a canonical fractional ideal of $R_H$.
Writing the elements of $v(K)$ in order from the smallest to the largest, first few terms are $0,1,7,9,10$ and $12$.
Assume $I$ is a nontrivial semidualizing module of $R$.
Then $I$ can be chosen as an ideal of $R_H$ (Remark \ref{r214}).
As $I$ is $2$-generated by Lemma \ref{dual}, write $I=(a,b)$ and assume $v(b)\ge v(a)$.
Replacing $b$ by $b-u\cdot a$ for a suitable unit $u\in k^\times$, we may assume $v(b)>v(a)$.
Then, via the multiplication by $1/a$, $I$ is isomorphic to $\frac{1}{a}I=(1,\frac{b}{a})$.
Replacing $I$ by $\frac{1}{a}I$, we can assume that $I$ is generated by $1$ and $f$, where $f=\frac{b}{a}$.
Note that $v(f)=v(b)-v(a)>0$.

Since $K:I \cong I^\vee$, the fractional ideal $K:I$ is also $2$-generated by Lemma \ref{dual}.
Suppose that $K:I$ is generated by $g$ and $h$ with $v(h)\ge v(g)$.
Since $K:I \subseteq K: (1)=K \subseteq k[\![t]\!]$, both $g$ and $h$ belong to $k[\![t]\!]$, namely, $v(g),v(h) \ge 0$.
Replacing $h$ by $h-u\cdot g$ for a suitable unit $u\in k^\times$, we may assume $v(h)>v(g)$.
By Remark \ref{r214}, $K=I(K:I)$, hence $K$ is generated by $g,h,fg,fh$ as an $R_H$-module.

\begin{claim} 
$v(g)=0$.
\end{claim}

\begin{proof}[Proof of Claim 1]
Since $1\in K$, there exists $\alpha,\beta,\gamma,\delta \in R_H$ such that $\alpha g+\beta h+ \gamma fg +\delta fh=1$.
Then $v(\alpha g+\beta h+ \gamma fg +\delta fh)=v(1)=0$.
Also $v(\alpha),v(\beta),v(\gamma), v(\delta) \ge 0$, $v(fh)>v(h)>v(g)$, and $v(fg)>v(g)\ge 0$.
Hence we get 
\[
0 \le v(g) \le v(\alpha g+\beta h+ \gamma fg +\delta fh)=0,
\]
which shows $v(g)=0$.
\end{proof}

\begin{claim}
$\min\{v(f),v(h)\}=1$ and $v(f)\not=v(h)$.
\end{claim}

\begin{proof}[Proof of Claim 2]
Since $t\in K$, there exists $\alpha,\beta,\gamma,\delta \in R_H$ such that $\alpha g+\beta h+ \gamma fg +\delta fh=t$.
Then $v(\alpha g+\beta h+ \gamma fg +\delta fh)=v(t)=1$.
Set $y':=\beta h+ \gamma fg +\delta fh$.
Assume $v(\alpha)=0$.
Then 
\[
v(\alpha g)=v(g)<\min\{v(\beta h),v(\gamma fg), v(\delta fh)\} \le v(y').
\]
Therefore, $v(y)=v(\alpha g)=0$, which contradicts to the equality $v(y)=1$.
Thus $v(\alpha)>0$.
As $v(R_H)=H$, this implies $v(\alpha)\ge 9$.
Since $1=v(y) \ge \min\{\alpha g, y'\} \text{ and } v(\alpha g) \ge 9$, one has $1 = v(y')$.
On the other hand, 
\[
v(y') \ge \min\{v(\beta h), v(\gamma fg), v(\delta fh)\} \ge \min\{v(h), v(fg)\}=\min\{v(h), v(f)\}>0.
\]
Hence we obtain $\min\{v(h),v(f)\}=1$.

Also, we can verify that $v(h)\not=v(f)$; otherwise, 
\[
2=v(h)+v(f)=v(hf)\in v(K),
\]
which leads to a contradiction.
\end{proof}

Now we show a contradiction.
Since $t^7\in K$, there exists  $\alpha,\beta,\gamma,\delta \in R_H$ such that $\alpha g+\beta h+ \gamma fg +\delta fh=t^7$.
Set $y':=\beta h+ \gamma fg +\delta fh$.
By the same reason as in the proof of Claim 2, $v(\alpha)\ge 9$.
Then, since $7=v(y) \ge \min\{\alpha g, y'\}$, $v(y')=7$.
We divide the proof into two cases.

Case 1: Suppose $v(f)>v(h)$.
By Claim 2, $v(h)=1$.
If $v(\beta)=0$, then 
\[
v(\beta h)=v(h)< \min\{v(\gamma fg), v(\delta fh)\} \le v(\gamma fg+\delta fh).
\]
Therefore $v(y')=v(\beta h)=1$, which contradicts to the equalitiy $v(y')=7$.
Thus $v(\beta)>0$.
As $v(R_H)=H$, this implies $v(\beta)\ge 9$.
Since $7=v(y') \ge \min\{\beta h,\gamma fg+\delta fh\}$ and $v(\beta h) \ge 9$, one has $7=v(\gamma fg+\delta fh)$.
On the other hand,
\[
7=v(\gamma fg+\delta fh) \ge \min\{ v(\gamma fg),v(\delta fg)\} \ge \min\{v(f),v(fh)\}=v(f)>v(h)=1.
\]
In view of $f\in K$, one can deduce $v(f)=7$.
It follows that $8=v(f)+v(h)=v(fh)\in v(K)$, which leads to a contradiction.

Case 2: Suppose $v(f)<v(h)$.
In this case, a similar argument to Case 1 also works.
Therefore, we achieve $8=v(f)+v(h)=v(fh)\in v(K)$, which leads to a contradiction.
\end{proof}

\begin{prop} \label{p45}
Let $H$ be one of the following numerical semigroups:
\begin{enumerate}[ \rm(1)]
\item $\langle9,12,15,19,20\rangle$ (the Sample of no.\ 15),
\item $\langle9,12,15,26,28\rangle$ (the Sample of no.\ 97),
\item $\langle9,12,15,19,26\rangle$ (the Sample of no.\ 100),
\item $\langle9,12,15,17,19\rangle$ (the Sample of no.\ 102),
\item $\langle9,15,17,19,21\rangle$ (the Sample of no.\ 103)
\end{enumerate}
Then $R_H$ has a nontrivial semidualizing module.
\end{prop}

\begin{proof}

By Theorem \ref{regular}, it is enough to show that $A_H$ has a nontrivial semidualizing module.

(1): $A_H \cong k[x,y,z,w]/(x^2,xy,y^2,z^2,zw,w^2)$ is isomorphic to a tensor product $k[x,y]/(x^2,xy,y^2) \otimes_k k[z,w]/(z^2,zw,z^2)$ of two non-Gorenstein local $k$-algebras.
So it has nontrivial semidualizing module by \cite[Corollary 4.9]{A}.

(2): $A_H$ is isomorphic to $k[x,y,z,w]/(x^2,xy,y^2,z^2,zw,w^2)$ which we consider in (1).
Therefore, it has a nontrivial semidualizing module.

(3): We write $R=R_H=k[\![t^9,t^{12},t^{15},t^{19},t^{26}]\!]$.
Then $K=(1, t^3, t^7, t^{10})$ is a canonical fractional ideal of $R$.
Let $I=(1,t^3)$.
Then $I^{\vee}=K: I=(1, t^7)$.
Note that $\mu_R(I)=\mu_R(I^{\vee})=2$, and $\mu_R(K)=4$.
One can check that $R:I=(t^9, t^{12}, t^{15})$ and $I^{\vee}:I=(t^9, t^{12}, t^{15}, t^{16}, t^{19}, t^{22})$. Hence we have $(R:I)I^{\vee}=I^{\vee}:I$ which implies $I\otimes_R I^{\vee}$ is torsion-free by Proposition~\ref{prop6.3}, and we get $\Ext_R^1(I, I)=0$ by \cite[Lemma 4.6]{HW}.
Now we show $\Ext^2_R(I, I)=0$.
Since $I^{\vee}=(1, t^7)$, one can apply Proposition~\ref{prop6.3} one more time to get $$T((R:I)\otimes_R I^{\vee})=((R:I):I^{\vee})/(R:I^{\vee})(R:I).$$
Furthermore we have $(R:I):I^{\vee}=(t^{21}, t^{24}, t^{27}, t^{28}, t^{31}, t^{34}, t^{35}, t^{38}, t^{41})$ and $R:I^{\vee}=(t^{12}, t^{19}, t^{26})$. Therefore $(R:I):I^{\vee}=(R:I^{\vee})(R:I)$ which implies that $(R:I)\otimes_R I^{\vee}$ is torsion-free.
As $\mu_R(I)=2$, applying \cite[Lemma 3.3]{HH}, we get a short exact sequence
\[
0 \to R:I \to R^{\oplus 2} \to I \to 0.
\]
This yields that $\Ext^2_R(I, I) \cong \Ext^1_R(R:I, I)$.
Now we get $\Ext^2_R(I, I)=0$ by \cite[Lemma 4.6]{HW}. Next note that we have the short exact sequence 
\begin{equation}
0\xrightarrow{} I \xrightarrow{t^9} R:I \xrightarrow{} R/ ((t^9, t^{12}):_R t^{15}) \xrightarrow{} 0 
\end{equation}
where $(t^9, t^{12}):_R t^{15}=R:I$.
Thus we get the long exact sequence 
\begin{equation}\cdots \xleftarrow{} \Ext_R^i(I, I)\xleftarrow{} \Ext_R^i(R:I, I) \xleftarrow{} \Ext^i_R(R/(R:I), I) \xleftarrow{} \cdots 
\end{equation}

Since $\Ext_R^i(R:I, I) \cong \Ext_R^{i+1} (I,I)$ and $\Ext^i_R(R/(R:I), I) \cong \Ext^{i-1}_R(R:I, I) \cong \Ext^i_R(I, I)$ for all $i\geq 2$, we have $\Ext^i_R(I, I) = 0$ for all $i \ge 1$. 

Finally we have $\Hom_R(I, I)=R$ by Proposition \ref{p219}.
We conclude that $I$ is a nontrivial semidualizing module of $R$.

(4): 
Let $H'=\langle9,12,15,19,26\rangle$, which is the semigroup we consider in (3).
Since 
\[
A_H \cong k[x,y,z,w]/(x^2,xy,y^2,yw-z^2,zw,w^2) \text{ and } A_{H'}=k[x,y,z,w]/(x^2,xy,xw-z^2,y^2,zw,w^2),
\]
one can see that $A_H$ is isomorphic to $A_{H'}$ via the permutation $(x,y,z,w) \mapsto (y,x,z,w)$ on the variables.
Thus, by (3), $A_H$ has a nontrivial semidualizing module.

(5): We write $R=R_H=k[\![t^9,t^{15},t^{17},t^{19},t^{21}]\!]$.
Then $K=(1, t^2, t^6, t^8)$ is a canonical fractional ideal of $R$.
Let $I=(1,t^2)$.
Then $I^{\vee}=K: I=(1, t^6)$.
Note that $\mu_R(I)=\mu_R(I^{\vee})=2$, and $\mu_R(K)=4$.
One can check that $R:I=(t^{15}, t^{17}, t^{19})$ and $I^{\vee}:I=(t^{15}, t^{17}, t^{19}, t^{21},t^{23},t^{25})$. Hence we have $(R:I)I^{\vee}=I^{\vee}:I$ which implies $I\otimes_R I^{\vee}$ is torsion-free by Proposition~\ref{prop6.3}, and we get $\Ext_R^1(I, I)=0$ by \cite[Lemma 4.6]{HW}.
Now we show $\Ext^2_R(I, I)=0$.
Since $I^{\vee}=(1, t^6)$, one can apply Proposition~\ref{prop6.3} one more time to get $$T((R:I)\otimes_R I^{\vee})=((R:I):I^{\vee})/(R:I^{\vee})(R:I).$$
Furthermore we have $(R:I):I^{\vee}=(t^{24}, t^{26}, t^{28}, t^{30}, t^{32}, t^{34}, t^{36}, t^{38}, t^{40})$ and $R:I^{\vee}=(t^{9}, t^{15}, t^{21})$. Therefore $(R:I):I^{\vee}=(R:I^{\vee})(R:I)$ which implies that $(R:I)\otimes_R I^{\vee}$ is torsion-free.
As $\mu_R(I)=2$, applying \cite[Lemma 3.3]{HH}, we get a short exact sequence
\[
0 \to R:I \to R^{\oplus 2} \to I \to 0.
\]
This yields that $\Ext^2_R(I, I) \cong \Ext^1_R(R:I, I)$.
Now we get $\Ext^2_R(I, I)=0$ by \cite[Lemma 4.6]{HW}.
Next note that we have the short exact sequence 
\begin{equation}
0\xrightarrow{} I \xrightarrow{t^{15}} R:I \xrightarrow{} R/ ((t^{15}, t^{17}):_R t^{19}) \xrightarrow{} 0 
\end{equation}
where $(t^{15}, t^{17}):_R t^{19}=R:I$.
Thus we get the long exact sequence 
\begin{equation}\cdots \xleftarrow{} \Ext_R^i(I, I)\xleftarrow{} \Ext_R^i(R:I, I) \xleftarrow{} \Ext^i_R(R/(R:I), I) \xleftarrow{} \cdots 
\end{equation}

Since $\Ext_R^i(R:I, I) \cong \Ext_R^{i+1} (I,I)$ and $\Ext^i_R(R/(R:I), I) \cong \Ext^{i-1}_R(R:I, I) \cong \Ext^i_R(I, I)$ for all $i\geq 2$, we have $\Ext^i_R (I, I) = 0$ for all $i\ge 1$.

Finally we have $\Hom_R(I,I)=R$ by Proposition \ref{p219}.
We conclude that $I$ is a nontrivial semidualizing module of $R$.
\end{proof}

Next we give the list of faces corresponding to $\Delta$ of no.\ 15, 17 ,97, 100, 102, 103 in Table \ref{table3}.
\begin{enumerate}
\item $F_1$ denotes the face with $\Delta_{F_1}=\{(1,3),(2,3),(1,6),(2,6)\}$, (Note that $\Delta_{F_1}$ is the $\Delta$ of no.\  15.)
\item $F_2$ denotes the face with $\Delta_{F_2}=\{(1,3),(2,2),(2,3),(1,6),(2,6)\}$, (Note that $\Delta_{F_2}$ is the $\Delta$ of no.\ 17.)
\item $F_3$ denotes the face with
$\Delta_{F_3}=\{(3,8),(1,3),(6,8),(1,6)\}$, (Note that $\Delta_{F_3}$ is the $\Delta$ of no.\ 97.)
\item $F_4$ denotes the face with
$\Delta_{F_4}=\{(1,1),(3,8),(1,3),(6,8),(1,6)\}$, (Note that $\Delta_{F_3}$ is the $\Delta$ of no.\ 100.)
\item $F_5$ denotes the face with
$\Delta_{F_5}=\{(3,8),(1,3),(6,8),(1,6),(8,8)\}$, (Note that $\Delta_{F_3}$ is the $\Delta$ of no.\ 102.)
\item $F_6$ denotes the face with
$\Delta_{F_6}=\{(1,1),(3,8),(1,3),(6,8),(1,6),(8,8)\}$, (Note that $\Delta_{F_6}$ is the $\Delta$ of no.\ 103.)
\end{enumerate}

We illustrate all faces which belong to orbits of $F_1,\dots,F_6$ in the next table.

\begin{table}[h]
\caption{}
\label{table4}
\begin{tabular}{|c|c|c|c|c|c|c|}\hline
\multirow{2}{*}{$\mathrm{Aut}(\mathbb{Z}/9\mathbb{Z})$} & \multicolumn{6}{|c|}{Faces}\\
\cline{2-7}
{}
& no.\ 15 & no.\ 17 & no.\ 97 & no.\ 100 & no.\ 102 & no.\ 103\\
\hline
1 & $F_1$ & $F_2$ & $F_3$ & $F_4$ & $F_5$ & $F_6$ \\
\hline
2 & $F_7$ & $F_8$ & $F_9$ & $F_{10}$ & $F_{11}$ & $F_{12}$\\
\hline
4 & $F_{13}$ & $F_{14}$ & $F_{15}$ & $F_{16}$ & $F_{17}$ & $F_{18}$\\
\hline
5 & $F_{19}$ & $F_{20}$ & $F_{15}$ & $F_{17}$ & $F_{16}$ & $F_{18}$\\
\hline
7 & $F_{21}$ & $F_{22}$ & $F_{9}$ & $F_{11}$ & $F_{10}$ & $F_{12}$\\
\hline
8 & $F_{23}$ & $F_{24}$ & $F_3$ & $F_5$ & $F_4$ & $F_6$\\
\hline
\end{tabular}
\end{table}

Here, the faces $F_7,\dots,F_{24}$ are defined as follows:

\begin{enumerate}
\setcounter{enumi}{6}
\item $F_7$ denotes the face with
$\Delta_{F_7}=\{(4,6),(2,3),(3,4),(2,6)\}$,
\item $F_8$ denotes the face with
$\Delta_{F_8}=\{(4,6),(2,3),(3,4),(2,6),(4,4)\}$,
\item $F_9$ denotes the face with
$\Delta_{F_9}=\{(3,7),(6,7),(2,3),(2,6)\}$,
\item $F_{10}$ denotes the face with
$\Delta_{F_{10}}=\{(3,7),(2,2),(6,7),(2,3),(2,6)\}$,
\item $F_{11}$ denotes the face with
$\Delta_{F_{11}}=\{(3,7),(6,7),(2,3),(7,7),(2,6)\}$,
\item $F_{12}$ denotes the face with
$\Delta_{F_{12}}=\{(3,7),(2,2),(6,7),(2,3),(7,7),(2,6)\}$,
\item $F_{13}$ denotes the face with
$\Delta_{F_{13}}=\{(4,6),(3,8),(6,8),(3,4)\}$,
\item $F_{14}$ denotes the face with
$\Delta_{F_{14}}=\{(4,6),(3,8),(6,8),(3,4),(8,8)\}$,
\item $F_{15}$ denotes the face with
$\Delta_{F_{15}}=\{(4,6),(5,6),(3,4),(3,5)\}$,
\item $F_{16}$ denotes the face with
$\Delta_{F_{16}}=\{(4,6),(5,6),(3,4),(3,5),(4,4)\}$,
\item $F_{17}$ denotes the face with
$\Delta_{F_{17}}=\{(4,6),(5,5),(5,6),(3,4),(3,5)\}$,
\item $F_{18}$ denotes the face with
$\Delta_{F_{18}}=\{(4,6),(5,5),(5,6),(3,4),(3,5),(4,4)\}$,
\item $F_{19}$ denotes the face with
$\Delta_{F_{19}}=\{(5,6),(1,3),(1,6),(3,5)\}$,
\item $F_{20}$ denotes the face with
$\Delta_{F_{20}}=\{(1,1),(5,6),(1,3),(1,6),(3,5)\}$,
\item $F_{21}$ denotes the face with
$\Delta_{F_{21}}=\{(3,7),(5,6),(6,7),(3,5)\}$,
\item $F_{22}$ denotes the face with
$\Delta_{F_{22}}=\{(3,7),(5,5),(5,6),(6,7),(3,5)\}$,
\item $F_{23}$ denotes the face with
$\Delta_{F_{23}}=\{(3,7),(3,8),(6,7),(6,8)\}$,
\item $F_{24}$ denotes the face with
$\Delta_{F_{24}}=\{(3,7),(3,8),(6,7),(7,7),(6,8)\}$.
\end{enumerate}
Note that each of the $F_i^{\circ}$ has a lattice point by \cite[Proposition 2.8]{K}. 

As a conclusion, we have the following main theorem.
\begin{theorem} \label{mainthm}
A numerical semigroup ring $R_H$ of multiplicity $9$ has a nontrivial semidualizing module if and only if $\mu(H)$ belongs to one of $F_i^{\circ}$ $(i=1,\dots,24)$.
\end{theorem}

\section{Existence in Higher Multiplicity}\label{section5}

Let $m\in \mathbb{N}$ and consider the numerical semigroups $H$ with multiplicity $e(R_H)=m$. 
In the case of $m=9$, Lemma~\ref{lemmafaces} gives us a way to identify which $\Delta$'s do not correspond to a face of Kunz’s polyhedron. 
As one moves from multiplicity $m=9$ to $m = 10$, the number of sets $\Delta \subseteq \{(i,j)\mid 1 \le i \le j \le m-1,i+j\not=m\}$ drastically increases. This raises the need for more efficient ways of identifying which sets $\Delta$ correspond to faces of Kunz’s polyhedron in higher multiplicity. However, in this section, we use the idea of ``gluing” in conjunction with Theorem 4.6 to provide examples of numerical semigroup rings in any multiplicity  possessing a nontrivial semidualizing module.

To understand the methods of this section, we introduce the notion of gluing. Gluing first appeared in the work of \cite{W73} and was later generalized by \cite{D76}.

\begin{definition}\label{def:gluing}
Let $A, B \subseteq \mathbb{N}$ such that $\langle A \rangle$ and $\langle B \rangle$ are two numerical semigroups. Consider integers $a \in \langle A \rangle$ and $b \in \langle B \rangle$ such that $\gcd(a,b) = 1$. Then the numerical semigroup $\langle b A, a B \rangle$ is called a \emph{gluing of $\langle A \rangle$ and $\langle B \rangle$} where $bA = \{ba: a \in A\}$ and $aB = \{ab: b \in B\}$.
\end{definition}

\begin{remark}\label{rem:nongen}
In much of the literature (see \cite{GS19,N13}), there is an extra condition that $a \notin \langle A \rangle \backslash A$ and $b \notin \langle B \rangle \backslash B$. This condition ensures that the union $b A \cup a B$ is a disjoint union \cite[Lemma 2.1]{GS19} of generators for $\langle b A, a B \rangle$. Our work does not depend on this union being disjoint, so we drop the extra condition on $a$ and $b$. A consequence of dropping this condition is that the corresponding Cohen presentation of the numerical semigroup ring is not necessarily minimal.
\end{remark}

Frequently, we take $B = \{1\}$ and thus the only restriction on the value of $b \in \langle B \rangle$ comes from $\gcd(a,b)=1$. As a result, we consider numerical semigroups coming from gluings of the form $ H' = \langle a , b H \rangle$ where $H = \langle A \rangle$. 
Next, we see how the relationship between $H$ and $H’$ extends to their corresponding numerical semigroup rings.

\begin{prop}[{\cite[Prop 2.2]{GS19}}] \label{prop:defIdeal}
Let $H = \langle a_1, \ldots, a_h\rangle$ be a numerical semigroup and $a = c_1 a_1 + \cdots + c_h a_h \in H $. Take $b \in \NN$ such that $\gcd(a,b) = 1$ and set $H' = \langle a , bH \rangle$. If $k \ldb H \rdb \cong k \ldb x_1, \ldots, x_h \rdb/ I_{H}$, then $$k \ldb H' \rdb \cong k \ldb x_1, \ldots, x_h, x_{h + 1} \rdb/ (I_{H}, f)$$ where $f = x_{h+1}^b - \prod_{i = 1}^h x_i^{c_i}$.
\end{prop}

\begin{remark}
The result in \cite[Prop 2.2]{GS19} differs slightly from the above. In their result, the authors consider gluings of the form $\langle bA, aB\rangle$ as described in Remark \ref{rem:nongen}. That is, they do not restrict $B = \{1\}$ but they do use the assumption that $a \in \langle A \rangle \backslash A$ and $b \in \langle B \rangle \backslash B$. However, the restrictions on $a$ and $b$ are not used in the proof, thus the result still holds for any $a \in \langle A \rangle$ and any $b \in \langle B \rangle$.
\end{remark}

\begin{cor}\label{cor:highermult}
Take $H$ and $H'$ as in Proposition \ref{prop:defIdeal}. If $k \ldb H \rdb$ has a nontrivial semidualizing module, then so does $k \ldb H' \rdb$.
\end{cor}

\begin{proof} Let $R=k \ldb H \rdb$ and observe that $k \ldb H’ \rdb \cong R[x_{h+1}]/(f)$ by Proposition~\ref{prop:defIdeal}. 
Since $f$ is a monic polynomial in $R[x]$, then $R[x]/(f)$ is a free $R$-module and one obtains a faithfully flat local homomorphism $k \ldb H \rdb \to k \ldb H' \rdb$. The desired result is then obtained by Theorem \ref{fflat}.
\end{proof}

\begin{theorem}\label{cor:allmult}
For $9 \leq a \in \NN$, there exists a numerical semigroup ring $R$ with $e(R) = a$ such that $R$ has a nontrivial semidualizing module.
\end{theorem}

\begin{proof}
Theorem \ref{mainthm} fully classifies $a = 9$. To address $a >9$, we consider the following cases.

{\bf Case 1:} Consider $a \notin \{13,14,16,17\}$. Take $H$ as in Example \ref{ex7.3}, that is $H = \langle 9, 10, 11, 12, 15 \rangle$. For all $a \in H$, one has $a \geq 9$ and $a \notin \{13, 14, 16, 17\}$. Now, given $a$, select $b$ such that $9b > a$ and $\gcd(a,b) = 1$. Set $H' = \langle a, bH \rangle$ and observe that $k \ldb H' \rdb$ has multiplicity $a$ and, by Corollary \ref{cor:highermult}, it has a nontrivial semidualizing module.

{\bf Case 2:} Consider $a = 13$. We take $H = \langle 9, 11, 12, 13, 15 \rangle$ and note that $\mu(H) \in F_8$. Thus, by Theorem \ref{mainthm}, $k \ldb H \rdb$ has a nontrivial semidualizing module. Take $a = 13$ and consider $b \in \NN$ such that $\gcd(13,b) = 1$ and $9b > 13$. Set $H' = \langle 13, bH \rangle$ and observe that $k \ldb H' \rdb$ has multiplicity 13 and, by Corollary \ref{cor:highermult}, it has a nontrivial semidualizing module.

{\bf Case 3:} Consider $a \in \{14, 16\}$. We take $H =  \langle 9, 12, 14, 15, 16 \rangle$ and note that $\mu(H) \in F_{22}$. Thus, by Theorem \ref{mainthm}, $k \ldb H \rdb$ has a nontrivial semidualizing module. Take $a = 14$ or $a = 16$ and consider $b \in \NN$ such that $\gcd(a,b) = 1$ and $9b > a$. Set $H' = \langle a, bH \rangle$ and observe that $k \ldb H' \rdb$ has multiplicity $a$ and, by Corollary \ref{cor:highermult}, it has a nontrivial semidualizing module.

{\bf Case 4:} Consider $a = 17$. We take $H = \langle 9, 12, 15, 17, 19 \rangle$ and note that $\mu(H) \in F_5$. Thus, by Theorem \ref{mainthm}, $k \ldb H \rdb$ has a nontrivial semidualizing module. Take $a = 17$ and consider $b \in \NN$ such that $\gcd(17,b) = 1$ and $9b > 17$. Set $H' = \langle 17, bH \rangle$ and observe that $k\ldb H' \rdb$ has multiplicity 17 and, by Corollary \ref{cor:highermult}, it has a nontrivial semidualizing module.
\end{proof}

Although the aforementioned proof establishes the existence of nontrivial semidualizing modules for certain numerical semigroup rings with higher multiplicity, it does not provide a concrete construction for such modules. The following proposition, however offers the first hint towards a construction of semidualizing modules that mirrors the gluing of numerical semigroups.

\begin{prop}\label{prop:gluCanon}
Let $A, B \subset \NN$ such that $\langle A \rangle$ and $\langle B \rangle$ are numerical semigroups with canonical modules $K_A$ and $K_B$, respectively. If $a \in \langle A \rangle$ and $b \in  \langle B \rangle$ such that $\gcd(a,b) = 1$, then the canonical module for $H = \langle bA, aB \rangle$ is given by \[K_H = \left( t^{br + as} \mid t^r \in K_A, t^s \in K_B \right).\]
\end{prop}

\begin{proof}
From \cite[Prop 6.6]{N13}, one gets that the set of pseudo-Frobenius numbers of $H = \langle bA, aB \rangle$ is given by the set \[\operatorname{PF}(H) = \{ bf + af' + ab : f \in \operatorname{PF}(A), f' \in \operatorname{PF}(B)\}\] where $\operatorname{PF}(A)$ and $\operatorname{PF}(B)$ are the sets of pseudo-Frobenius numbers of $\langle A \rangle$ and $\langle B \rangle$, respectively. From here (or using \cite[Prop 10(i)]{D76}), one notes that the Frobenius number of $H$ is given by \[F_H = b F_A + a F_B + ab\] where $F_A$ and $F_B$ are the Frobenius numbers of $\langle A \rangle$ and $\langle B \rangle$. Finally, we note from \cite{N13} that $K_H = (t^{F_H - f} : f \in \operatorname{PF}(H))$ which, when combined with the above observations, gives the desired result.
\end{proof}

While this proposition addresses the canonical module rather than a nontrivial semidualizing module, it suggests how one might construct a nontrivial semidualizing module. In fact, the next theorem shows that the construction of a semidualizing module for a gluing resembles the construction of the canonical module.

\begin{theorem}\label{thm:semidualglue}
Consider the set-up of Proposition \ref{prop:gluCanon}. If the numerical semigroup rings  for $\langle A \rangle$ and $\langle B \rangle$ have semidualizing modules $I_1$ and $I_2$ (generated by monomials), respectively, then the numerical semigroup ring $k \ldb H \rdb$ has semidualizing module
\[I=\left( t^{br + as} \mid t^r \in I_1, t^s \in I_2 \right).\]

Consequently, if the numerical semigroup ring for $\langle A \rangle$ or $\langle B \rangle$ has a nontrivial semidualizing module, then so does $k \ldb H \rdb$.
\end{theorem}

\begin{proof}
We start by fixing the following notation. We let $A = \{a_1,\ldots,a_n\}$ and $B = \{b_1,\ldots,b_m\}$. For $a \in \langle A \rangle$ and $b \in \langle B \rangle$ we have $r_1,\ldots,r_n,s_1,\ldots,s_m \in \ZZ$ such that \[a = \sum_{i = 1}^n r_i a_i \quad \hbox{ and } \quad b = \sum_{j = 1}^m s_j b_j.\] From here we set $I_A$ and $I_B$ to be the defining ideals of $k \ldb H_A \rdb$ and $k \ldb H_B \rdb$. That is, we have ideal $I_A \subseteq k \ldb x_1, \ldots, x_n \rdb$ and $I_B \subseteq k \ldb y_1, \ldots, y_m \rdb$ such that $k \ldb H_A \rdb \cong k \ldb x_1, \ldots, x_n \rdb/ I_A$ and $k \ldb H_B \rdb \cong k \ldb y_1, \ldots, y_m \rdb/ I_B$. Lastly, we set \[f = \prod_{i = 1}^n x_i^{r_i} - \prod_{j = 1}^m y_j^{s_j} \in k \ldb x_1,\ldots,x_n,y_1,\ldots,y_m\rdb.\]

We note that $I_1$ (and $I_2$) is semidualizing over $k \ldb H_A \rdb$ (respectively $k \ldb H_B \rdb$) if and only if it is semidualizing over $k[H_A] = k[x_1,\ldots,x_n]/I_A$ (respectively over $k[y_1,\ldots,y_m] / I_B$). Due to \cite[Theorem 1.2(a)]{A}, we have $I_1 \otimes_k I_2$ is a semidualizing module of $R = k[H_A] \otimes_k k[H_b] \cong k [ x_1, \ldots, x_n, y_1, \ldots, y_m ] / (I_A + I_B)$. Since $\widehat{I_1 \otimes_k I_2} \cong \left( I_1 \otimes_k I_2 \right) \otimes_{R} \widehat{ R }$, \cite[Lemma 2.6]{C} tells us that $\widehat{I_1 \otimes_k I_2}$ is a semidualizing $\widehat{R}$-module. 

Next, we note that $\widehat{R} \cong k \ldb x_1, \ldots, x_n, y_1, \ldots, y_m \rdb / (I_A + I_B)$ and \cite[Proposition 2.2]{GS19} yields $k \ldb H \rdb \cong \widehat{R} / (f)$. The above citation supposes the conditions mentioned in Remark \ref{rem:nongen}, but a minimal Cohen presentation is not necessary to obtain the second isomorphism above. Since $\widehat{R}$ is a domain, $f$ is a regular element and, by \cite[Proposition 5.8]{C}, we have $\widehat{I_1 \otimes_k I_2} \otimes_{\widehat{R}} \left(\widehat{R} / (f)\right)$ is a semidualizing $\widehat{R} / (f)$-module. To use this result, one observes that $\widehat{I_1 \otimes_k I_2} \otimes_{\widehat{R}} \left(\widehat{R} / (f)\right) \cong (I_1 \otimes_k I_2) \otimes_R \left(\widehat{R} / (f)\right)$ as $\widehat{R} / (f)$-modules.

To conclude our proof, we will show that the module $I$ given in the statement of the theorem satisfies $I \cong (I_1 \otimes_k I_2) \otimes_R \left(\widehat{R} / (f)\right)$. To obtain this isomorphism, we identify $f$ with $t^a \otimes 1 - 1 \otimes t^b$ which is its image in $R$. This gives the isomorphism \[(I_1 \otimes_k I_2) \otimes_R \left(\widehat{R} / (f)\right) \cong \frac{I_1 \otimes_k I_2}{( t^a \otimes 1 - 1 \otimes t^b ) \cdot (I_1 \otimes_k I_2)} \otimes_R \widehat{R}.\] The map $\varphi : \frac{I_1 \otimes_k I_2}{( t^a \otimes 1 - 1 \otimes t^b ) \cdot (I_1 \otimes_k I_2)} \otimes_R \widehat{R} \to I$ given on simple tensors by $\varphi(t^{\ell} \otimes t^s \otimes r) = r(t^{\ell})^b (t^s)^a$ is surjective by definition of $I$. Since both modules are torsion free and of the same rank, we have a $\widehat{R}/(f)$-module isomorphism, producing the desired result.
\end{proof}

We now exhibit the results of Theorem \ref{cor:highermult}, Proposition \ref{prop:gluCanon}, and Theorem \ref{thm:semidualglue} all in a single example.

\begin{example}\label{ex:highmult}
To give concrete examples of Theorem\ref{cor:allmult}, we consider multiplicity $10 \leq m \leq 18$. We start by setting the following notation.
	\begin{enumerate}
		\item $H_1 = \langle 9, 10, 11, 12, 15 \rangle$ has canonical module $K = (1, t, t^3, t^4)$ and semidualizing module $I = (1,t)$.
		\item  $H_2 = \langle 9, 11, 12, 13, 15 \rangle$ has canonical module $K = (1,t^2,t^3,t^5)$ and semidualizing module $I = (1,t^2)$.
		\item $H_3 = \langle 9, 12, 14, 15, 16 \rangle$ has canonical module $K = (1, t^2, t^3, t^5)$ and semidualizing module $I = (1, t^2)$.
		\item $H_4 = \langle 9, 12, 15, 17, 19 \rangle$ has canonical module $K = (1,t^2,t^3,t^5)$ and semidualizing module $I = (1,t^2)$.
	\end{enumerate}
The listed semidualizing module for $H_1$ was computed in Example \ref{ex7.3}. The semidualizing modules listed for $H_2$, $H_3$, and $H_4$ were computed using the same techniques. We now list out our examples in higher multiplicity obtained from gluings from the above list.

\begin{longtable}{|c|c|c|c|c|}
\caption{}
\label{tablemult}
\\
\hline
Multiplicity & $H$ & Generators & $K_H$ & $I$ \\
\hline
10 & $\langle 10, 3H_1 \rangle$ & $\{10,27,33,36,45\}$ & $(1,t^3,t^9,t^{12})$ & $(1,t^3)$ \\
\hline
11 & $\langle 11, 2H_1 \rangle$ & $\{11,18,20,24,30\}$ & $(1,t^2,t^6,t^8)$ & $(1,t^2)$ \\
\hline
12 & $\langle 12, 5H_1 \rangle$ & $\{12,45,50,55,75\}$ & $(1,t^5,t^{15},t^{20})$ & $(1,t^5)$ \\
\hline
13 & $\langle 13, 2H_2\rangle$ & $\{13,18,22,24,30\}$ & $(1,t^4,t^6, t^{10})$ & $(1,t^4)$ \\
\hline
14 & $\langle 14, 3H_3 \rangle$ & $\{14,27,36,45,48\}$ & $(1,t^6,t^9,t^{15})$ & $(1,t^6)$ \\
\hline
15 & $\langle 15, 2H_1\rangle$ & $\{15,18,20,22,24\}$ & $(1,t^2,t^6,t^8)$ & $(1,t^2)$ \\
\hline
16 & $\langle 16, 3H_3 \rangle$ & $\{16,27,36,42,45\}$ & $(1,t^6, t^9, t^{15})$ & $(1,t^6)$ \\
\hline
17 & $\langle 17, 2H_4 \rangle$ & $\{17,18,24,30,38\}$ & $(1,t^4,t^6,t^{10})$ & $(1,t^4)$ \\
\hline
18 & $\langle 18, 5H_1 \rangle$ & $\{18,45,50,55,60,75\}$ & $(1,t^5,t^{15},t^{20})$, & $(1,t^5)$ \\
\hline
\end{longtable}

\end{example}

\begin{remark}
To understand Proposition \ref{prop:gluCanon} in our context, it is important to note that if $B = \{1\}$, then $F_B = -1$. One way to interpret this comes from the associated numerical semigroup ring $k \ldb t \rdb$. In this case the largest power of $t$ not in the semigroup ring is $t^{-1}$.
\end{remark}

It is worth noting that in Section 6 of \cite{N13}, the author assumed that $a \in \langle A \rangle \backslash A$ and $b \in \langle B \rangle \backslash B$. However, this additional assumption is unnecessary to obtain their result. Consequently, we use it freely. This, combined with Example \ref{ex:highmult}, provides the motivation for the following question.

\begin{question}\label{question:monomial}
Suppose $k \ldb H \rdb$ is a numerical semigroup ring with a nontrivial semidualizing module $I$. Can one always say that $I$ is generated by monomials?
\end{question}

\section*{Acknowledgments}

The authors would like to thank J\"urgen Herzog for his feedback and insightful discussions during his visit to Morgantown in April 2023 that led to the addition of Section \ref{section5}. Part of this work was completed when Kobayashi visited West Virginia University in September 2022 and March 2023. He is grateful for the kind hospitality of the WVU School of Mathematical and Data Sciences.

%\textcolor{red}{Things to Do} 	
%\textcolor{red}{ 	
%\begin{itemize} 		
%\item Omit Definition 2.12 and embed it in the Proof of Proposition 2.13. 	
%Maybe we can say that: ``Here $I_1(X)$ denotes the ideal generated by the $1\times 1$ minors of the presentation matrix of $X$." 	
%\end{itemize} 	
%} 

\end{document}